\documentclass[12pt, twoside]{article}
\usepackage{amsmath,amsthm,amssymb}
\usepackage{times}
\usepackage{enumerate}
\usepackage{mathrsfs}
\usepackage[all]{xy}

\def\titlerunning#1{\gdef\titrun{#1}}
\makeatletter
\def\author#1{\gdef\autrun{\def\and{\unskip, }#1}\gdef\@author{#1}}
\def\address#1{{\def\and{\\\hspace*{18pt}}\renewcommand{\thefootnote}{}%
\footnote {#1}}%
\markboth{\autrun}{\titrun}}
\makeatother
\def\email#1{e-mail: #1}
\def\subjclass#1{{\renewcommand{\thefootnote}{}%
\footnote{\emph{Mathematics Subject Classification (2010):} #1}}}
\def\keywords#1{\par\medskip
\noindent\textbf{Keywords.} #1}

\newtheorem{thm}{Theorem}[section]

\newtheorem{prop}[thm]{Proposition}

\theoremstyle{definition}

\newtheorem{rem}[thm]{Remark}

\numberwithin{equation}{section}

\begin{document}


\baselineskip=17pt


\titlerunning{}

\title{Twistor geometry of Hermitian surfaces induced by canonical connections}

\author{Jixiang Fu
\and
Xianchao Zhou}

\date{}

\maketitle

\address{J. Fu: School of Mathematical Sciences, Fudan University, Shanghai 200433, China; \email{majxfu@fudan.edu.cn}
\and
X. Zhou: Department of Applied Mathematics, Zhejiang University of Technology, Hangzhou 310023, China; \email{zhouxianch07@zjut.edu.cn}}



\begin{abstract}
In this paper, following the constructions of N. R. O'Brian, J. H. Rawnsley and I. Vaisman, we  define four almost Hermitian structures (up to conjugation) on the twistor space of a Hermitian surface by using canonical connections, including the Lichnerowicz connection and the Chern connection. We also study the relations between the natural geometry of the twistor spaces and the geometry of Hermitian surfaces.

\end{abstract}

\section{Introduction}
The twistor construction is an important technique in differential geometry and mathematical physics. This approach was first proposed by R. Penrose in 1960s. In 1978, the Riemannian version of R. Penrose's twistor programme was presented by M. F. Atiyah, N. J. Hitchin and I. M. Singer\cite{AHS}.

Basically, to each oriented Riemannian 4-manifold $M$, one can associate canonically a 6-manifold $Z$ (called the twistor space of $M$). $Z$ naturally admits
an almost complex structure $\mathbb{J}_+$ which preserves the decomposition of tangent bundle $TZ$ into horizontal component and vertical component, with respect to the Levi-Civita connection of Riemannian 4-manifold $M$. M. F. Atiyah, N. J. Hitchin and I. M. Singer\cite{AHS} proved that $\mathbb{J}_+$ is integrable if and only if $M$ is self-dual. Therefore, they established an elegant correspondence between  Yang-Mills fields on 4-manifolds and holomorphic vector bundles on complex 3-manifolds. On the other hand, in 1985, J. Eells and S. Salamon\cite{ES} introduced another almost complex structure $\mathbb{J}_-$ which, by contrast with $\mathbb{J}_+$, is never integrable. However, $\mathbb{J}_-$ plays an important role in the theory of harmonic map.

In the years since then, there are many results related to the generalizations of the twistor theory of Riemannian 4-manifold. S. Salamon\cite{Sa1} gave the generalized twistor space $Z$ of a Quaternionic K\"{a}hler  manifold $M$,  whose fibre $Z_x$ over a point $x\in M$ is the space of all almost complex structures on $T_x M$ compatible with the $Sp(n)Sp(1)$-structure. He also showed that  the twistor space $Z$ has a natural complex structure. Motivated by these examples, L. Berard Bergery and T. Ochiai\cite{BO} proposed a general theory of the twistor space $Z$ from the point of view of $G$-structure, and considered a natural almost complex structure on $Z$ with respect to a $G$-connection.  N. R. O'Brian and J. H. Rawnsley\cite{OR} studied  the integrability of natural almost complex structures on certain fibre bundles, which generalize the 4-dimensional twistor theory to arbitrary even dimension, by using the representation theory. As an important example, they studied the Grassmann bundles of the tangent bundle of an almost Hermitian manifold. R. Bryant\cite{Bry} developed a method for constructing holomorphic twistor spaces over Riemannian symmetric spaces of even dimension. C. K. Peng and Z. Z. Tang\cite{PT} gave a detailed description
of the twistor space over an oriented even dimensional Riemannian manifold by using the method of moving frames.

Of course, the interplay between 4-dimensional conformal geometry and the corresponding twistor geometry is one of the most attractive topics. There are some interesting applications.  By using examples of compact complex 3-manifolds that arise as twistor spaces, F. Campana\cite{Cam},
C. LeBrun and Y. S. Poon\cite{LP} proved that the class $\mathscr{C}$ of compact complex manifolds, bimeromorphic to K\"{a}hler manifolds, is not stable under small deformations of complex structure. J. Fine and D. Panov\cite{FP}  introduced the concept of definite connection on $SO(3)$-bundle over an oriented 4-manifold. They showed many non-K\"{a}hler symplectic Calabi-Yau 6-manifolds from the twistor spaces of Riemannian 4-manifolds \cite{FP}. There are also many studies with respect to the metric properties and curvature properties on the twistor space of Riemannian 4-manifold. For example, N. J. Hitchin\cite{Hit} showed that if the twistor space $(Z, \mathbb{J}_+)$ of  a compact self-dual 4-manifold $M$ admits a K\"{a}hler metric, then $M$ is the 4-sphere $S^4$ or the complex projective plane $\mathbb{C}P^2$. In fact, there is a  natural 1-parameter family of Riemannian metrics $g_t$ on the twistor space $Z$. Meanwhile, $\mathbb{J}_+$ and $\mathbb{J}_-$ are orthogonal almost complex structures with respect to the metrics $g_t$. Thus, it is natural to study the relations between the almost Hermitian geometry of $Z$ and the Riemannian geometry of 4-manifold $M$ \cite{FK,Mu,JR,DM,BN,DGM,FZ,DGM}. By using the method of moving frames, J. X. Fu and X. C. Zhou\cite{FZ} systematically studied special metric conditions (including the balanced metric condition, first Gauduchon metric condition\cite{FWW}) on almost Hermitian twistor spaces $(Z,g_t,\mathbb{J}_{\pm})$,  which are used to characterize  Riemannian 4-manifolds. Recently, M. Verbitsky\cite{Ver} obtained a generalization of Hitchin's theorem on the K\"{a}hler twistor space. He proved that if the twistor space $(Z, \mathbb{J}_+)$ of  a compact self-dual 4-manifold $M$ admits an SKT metric, then $M$ admits a K\"{a}hler metric. Hence $M$ is the 4-sphere $S^4$ or the complex projective plane $\mathbb{C}P^2$.

The main concern of this paper is the twistor geometry of Hermitian surfaces. There are two motivations. One is that for a Hermitian surface $M$, its twistor space $Z$ can be identified with the projective bundle $\mathbb{P}(T^{1,0}M)$. Meanwhile, following the ideas of N. R. O'Brian, J. H. Rawnsley \cite{OR} and I. Vaisman \cite{Va2},  we can define more geometric structures on the twistor space $Z$ by using canonical connections on $M$. We observe that these constructions coincide with the familiar geometry of the flag manifold $\frac{SU(3)}{S(U(1)\times U(1)\times U(1))}$, which can be considered as the twistor space of complex projective plane $\mathbb{C}P^2$. The other motivation is that we want to find more relations  between the (almost) complex geometry of the twistor spaces and the geometry of Hermitian surfaces. In this paper, we give a comprehensive study of the four natural almost Hermitian
structures on the twistor space $Z=\mathbb{P}(T^{1,0}M)$ associated with the Lichnerowicz connection and also with the Chern connection. For the Lichnerowicz connection, the induced natural almost Hermitian structures on the twistor space $Z=\mathbb{P}(T^{1,0}M)$ are denoted by $(\mathbb{J}_i^L, \mathbb{K}_i^L(\lambda))$, $i=1,2,3,4$. We prove that $\mathbb{J}_3^L$ (or $\mathbb{J}_4^L$) is integrable if and only if the Ricci tensor of the Levi-Civita connection on  $(M, J, h)$ is $J$-invariant (in Proposition 3.1). The symplectic metric condition and the balanced metric condition of the natural almost Hermitian metrics $\mathbb{K}_i^L(\lambda)$ on $Z=\mathbb{P}(T^{1,0}M)$ are studied (in Theorem 3.4 and Theorem 3.5).  For the Chern connection, the induced natural almost Hermitian structures on the twistor space $Z=\mathbb{P}(T^{1,0}M)$ are denoted by $(\mathbb{J}_i^{Ch}, \mathbb{K}_i^{Ch}(\lambda))$, $i=1,2,3,4$. We prove that $\mathbb{J}_1^{Ch}=\mathbb{J}_1^L$ (in Proposition 4.1, in fact as showed in Remark 4.2, a family of canonical Hermitian connections induce the same $\mathbb{J}_1$), and all $\mathbb{J}_i^{Ch}$ are conformally invariant (in Proposition 4.3). We also consider the balanced metric condition of the natural almost Hermitian metrics $\mathbb{K}_i^{Ch}(\lambda)$ on $Z=\mathbb{P}(T^{1,0}M)$ (in Theorem 4.5).

The paper is organized as follows. In Section 2, we give some essential facts of the geometry of a Hermitian surface $M$, and  P. Gauduchon's 1-parameter family $D^t$ of canonical Hermitian connections on a Hermitian manifold. Vaisman's constructions of  almost Hermitian
structures on the twistor space $Z=\mathbb{P}(T^{1,0}M)$ associated with any unitary connection are also presented. In Section 3, we focus on the twistor geometry  induced by the Lichnerowicz connection. In Section 4, we go on to study the twistor geometry  induced by the Chern connection.  In Section 5, we give some related discussions, including the natural Hermitian structure on the projective bundle (this structure often used in complex algebraic geometry), and G. Deschamps's works on the twistor geometry of Riemannian 4-manifolds. Finally, in the appendix, we make some explicit calculations  for the twistor geometry of complex projective plane $\mathbb{C}P^2$ with the Fubini-Study metric $h_{FS}$.

\section{Preliminaries and notations: Hermitian surfaces and twistor constructions}
Let $M$ be an oriented 4-manifold with a smooth Riemannian metric $h$. The Hodge star operator gives a map
$\ast:\wedge^2(T^\ast M)\rightarrow \wedge^2(T^\ast M)$ with $\ast^2=1$. Accordingly, its eigenvalues are
$\pm1$ and the bundle of two-forms splits
\begin{equation}
\wedge^2(T^\ast M)=\wedge^+(T^\ast M)\oplus \wedge^-(T^\ast M)
\end{equation}
into two eigenspaces. $\wedge^+(T^\ast M)$ (resp. $\wedge^-(T^\ast M)$) is called the bundle of self-dual (resp. anti-self-dual) 2-forms. This decomposition is
conformally invariant with respect to the Riemannian metric $h$. However, reversing the orientation of $M$ interchanges $\wedge^+(T^\ast M)$ and $\wedge^-(T^\ast M)$.

Using the metric $h$, we have the identifications of $\wedge^2(T^\ast M)=\wedge^2(TM)$ and $\wedge^2(TM)=\wedge^+(TM)\oplus \wedge^-(TM)$. Thus the Riemannian curvature tensor $R$ of $(M, h)$ can be considered as a self-adjoint operator
$\hat{R}:\wedge^2(T^\ast M)\rightarrow \wedge^2(T^\ast M)$ and so, with respect to the decomposition (2.1), it decomposes into parts
\begin{equation}
\hat{R}=\left(\begin{array}{cc}W^{+}+\frac{s}{12}\text{Id}&\text{Ric}_0^\ast\\\text{Ric}_0&W^{-}+\frac{s}{12}\text{Id}
\end{array}\right),
\end{equation}
where $W^+$ (resp. $W^-$) is the
self-dual (resp. anti-self-dual) Weyl curvature operator, $s$ is the scalar curvature, $\text{Ric}_0$ is the trace-free Ricci curvature operator, and
 $\text{Ric}_0^\ast$ is the transpose of $\text{Ric}_0$\cite{Bes}.

An oriented Riemannian 4-manifold $(M,h)$ is called self-dual (resp. anti-self-dual) if $W^-=0$ (resp. $W^+=0$). It is well-known that $(M,h)$ is an Einstein manifold if and only of $\text{Ric}_0=0$.

In this paper, we especially consider a Hermitian surface $(M, J, h)$, i.e. a Hermitian manifold of real dimension 4 with a complex structure $J$ and a compatible Riemannian metric $h$. Let $F$ be the fundamental 2-form defined by $F(X,Y)=h(JX,Y)$, for any tangent vector fields $X$ and $Y$ on $M$.  An interesting differential form of
$(M, J, h)$ is the Lee form $\alpha=J\delta F=-\delta F\circ J$, where $\delta$ denotes the codifferential with respect to the metric $h$. In this case, we have
\begin{eqnarray}
&&\wedge^-(T^\ast M)=[\wedge^{1,1}_0(T^\ast M)],\notag\\
&&\wedge^+(T^\ast M)=[[\wedge^{2,0}(T^\ast M)]]\oplus \mathbb{R} F,
\end{eqnarray}
where$[[\wedge^{2,0}(T^\ast M)]]\otimes_\mathbb{R} \mathbb{C}=\wedge^{2,0}\oplus \wedge^{0,2}$,  $[\wedge^{1,1}(T^\ast M)]\otimes_\mathbb{R} \mathbb{C}=\wedge^{1,1}$,
and $[\wedge^{1,1}_0(T^\ast M)]$ is the orthogonal complement of
$F$ in $[\wedge^{1,1}(T^\ast M)]$. These notations  refer to S. Salamon's book\cite{Sa2}.

For convenience, we give local orthonormal basis of these vector bundles. Let $(e_1, e_2=Je_1, e_3, e_4=Je_3)$ be a local $J$-adapted orthonormal frame (sometimes, called unitary frame) on $M$, its dual is denoted by $(\vartheta^1, \vartheta^2, \vartheta^3, \vartheta^4)$. Then a local orthonormal basis of $\wedge^\pm(TM)$ is given by
\begin{eqnarray}
&&\epsilon_1^\pm=\frac{1}{\sqrt{2}}(e_1\wedge e_2\pm e_3\wedge e_4),\\
&&\epsilon_2^\pm=\frac{1}{\sqrt{2}}(e_1\wedge e_3\pm e_4\wedge e_2),\\
&&\epsilon_3^\pm=\frac{1}{\sqrt{2}}(e_1\wedge e_4\pm e_2\wedge e_3),
\end{eqnarray}
respectively.
Dually, for the bundles $\wedge^\pm(T^\ast M)$, a local orthonormal basis is given by
\begin{eqnarray}
&&\alpha_\pm^1=\frac{1}{\sqrt{2}}(\vartheta^1\wedge\vartheta^2\pm\vartheta^3\wedge\vartheta^4),\\
&&\alpha_\pm^2=\frac{1}{\sqrt{2}}(\vartheta^1\wedge\vartheta^3\pm\vartheta^4\wedge\vartheta^2),\\
&&\alpha_\pm^3=\frac{1}{\sqrt{2}}(\vartheta^1\wedge\vartheta^4\pm\vartheta^2\wedge\vartheta^3),
\end{eqnarray}
respectively.

Set
\begin{eqnarray*}
&&u_1=\frac{1}{\sqrt{2}}(e_1-\sqrt{-1}e_2),~~u_2=\frac{1}{\sqrt{2}}(e_3-\sqrt{-1}e_4),\\
&&\eta^1=\frac{1}{\sqrt{2}}(\vartheta^1+\sqrt{-1}\vartheta^2),~~\eta^2=\frac{1}{\sqrt{2}}(\vartheta^3+\sqrt{-1}\vartheta^4).
\end{eqnarray*}
Then we have
\begin{eqnarray}
&&\epsilon_1^+=\frac{\sqrt{-1}}{\sqrt{2}}(\overline{u_1}\wedge u_1+\overline{u_2}\wedge u_2),
~~\epsilon_1^-=\frac{\sqrt{-1}}{\sqrt{2}}(\overline{u_1}\wedge u_1-\overline{u_2}\wedge u_2),\\
&&\epsilon_2^+=\frac{1}{\sqrt{2}}(u_1\wedge u_2+\overline{u_1}\wedge \overline{u_2}),
~~~~~\epsilon_2^-=\frac{1}{\sqrt{2}}(u_1\wedge \overline{u_2}+\overline{u_1}\wedge u_2),\\
&&\epsilon_3^+=\frac{\sqrt{-1}}{\sqrt{2}}(u_1\wedge u_2-\overline{u_1}\wedge \overline{u_2}),
~~\epsilon_3^-=\frac{\sqrt{-1}}{\sqrt{2}}(\overline{u_1}\wedge u_2-u_1\wedge\overline{u_2}),
\end{eqnarray}
and dually,
\begin{eqnarray}
&&\alpha_+^1=\frac{\sqrt{-1}}{\sqrt{2}}(\eta^1\wedge\overline{\eta^1}+\eta^2\wedge\overline{\eta^2}),
~~\alpha_-^1=\frac{\sqrt{-1}}{\sqrt{2}}(\eta^1\wedge\overline{\eta^1}-\eta^2\wedge\overline{\eta^2}),\\
&&\alpha_+^2=\frac{1}{\sqrt{2}}(\eta^1\wedge\eta^2+\overline{\eta^1}\wedge\overline{\eta^2}),
~~~~~\alpha_-^2=\frac{1}{\sqrt{2}}(\eta^1\wedge\overline{\eta^2}+\overline{\eta^1}\wedge\eta^2),\\
&&\alpha_+^3=\frac{\sqrt{-1}}{\sqrt{2}}(\overline{\eta^1}\wedge\overline{\eta^2}-\eta^1\wedge\eta^2),
~~\alpha_-^3=\frac{\sqrt{-1}}{\sqrt{2}}(\eta^1\wedge\overline{\eta^2}-\overline{\eta^1}\wedge\eta^2).
\end{eqnarray}

Next, we shall introduce the definition of  canonical Hermitian connections on  Hermitian manifolds.

A Hermitian connection (sometimes also called unitary connection) on a Hermitian manifold $(M, J, h)$ is a connection in the bundle $Q=U(M)$ of unitary frames on $M$, that is, a linear connection which is metric ($h$ is parallel) and complex  ($J$ is parallel). The existence of such connections is assured by the connection theory
in principal bundles\cite{KN}. P. Gauduchon\cite{Gau} introduced a 1-parameter family $D^t$ of canonical Hermitian connections on Hermitian manifold as follows.
\begin{align}\label{E:1}
h(D^t_{X_1} X_2, X_3) &= h(\nabla_{X_1} X_2, X_3)+\frac{1}{4}[dF(JX_1, JX_2, JX_3)-dF(JX_1, X_2, X_3)]\notag\\
           &\quad ~~~~~~~~~~~~~~~~~~~~~~~~~~-\frac{t}{4}[dF(JX_1, JX_2, JX_3)+dF(JX_1, X_2, X_3)],
\end{align}
where $\nabla$ is the Levi-Civita connection on $(M, J, h)$, $X_1, X_2$ and $X_3$ are tangent vector fields on $M$. When $h$ is a  K\"{a}hler metric, i.e. $dF=0$, all $\{D^t\}$ equal to  $\nabla$.

In this family, $D^0$ is the Lichnerowicz connection, $D^1$ is the Chern connection, $D^{-1}$ is the Bismut connection. Let $R$, $K$ and $\tilde{K}$ be the curvature tensor of $\nabla$, $D^1$ and $D^{-1}$, respectively. Set
\begin{align}\label{E:1}
R(X_1, X_2, X_3, X_4) &= h(R(X_3, X_4)X_2, X_1)\notag\\
                      &=h(\nabla_{X_3}\nabla_{X_4}X_2-\nabla_{X_4}\nabla_{X_3}X_2-\nabla_{[X_3,X_4]}X_2, X_1),
\end{align}
and similar for curvature tensors $K$ and $\tilde{K}$. There are many studies of curvature properties with respect to these canonical Hermitian connections, refer to\cite{LY, YZ} and the references therein.

For a Hermitian surface, by direct calculations, we get the following useful curvature relations\cite{Va1,IP}:
\begin{align}\label{E:1}
K(X_1, X_2, X_3, X_4) &= R(X_1, X_2, X_3, X_4)\notag\\
                      &~~~~+\frac{1}{2}d(\alpha\circ J)(X_3, X_4)F(X_1, X_2)\notag\\
                      &~~~~+\frac{1}{2}[L(X_4, X_2)h(X_3, X_1)+L(X_3, X_1)h(X_4, X_2)]\notag\\
                      &~~~~-\frac{1}{2}[L(X_3, X_2)h(X_4, X_1)+L(X_4, X_1)h(X_3, X_2)]\notag\\
                      &~~~~+\frac{|\alpha|^2}{4}[h(X_3, X_2)h(X_4, X_1)-h(X_4, X_2)h(X_3, X_1)],
\end{align}
\begin{align}\label{E:1}
\tilde{K}(X_1, X_2, X_3, X_4) &= R(X_1, X_2, X_3, X_4)\notag\\
                      &~~~~+\frac{1}{2}(\nabla_{X_3}(\alpha\circ J\wedge F))(X_4, X_2, X_1)\notag\\
                      &~~~~-\frac{1}{2}(\nabla_{X_4}(\alpha\circ J\wedge F))(X_3, X_2, X_1)\notag\\
                      &~~~~+\frac{1}{4}\sum _{i=1}^{4}(\alpha\circ J\wedge F)(X_4, X_1, e_i)(\alpha\circ J\wedge F)(X_3, X_2, e_i)\notag\\
                      &~~~~-\frac{1}{4}\sum _{i=1}^{4}(\alpha\circ J\wedge F)(X_3, X_1, e_i)(\alpha\circ J\wedge F)(X_4, X_2, e_i),
\end{align}
where $L(X, Y)=(\nabla_X \alpha)Y+\frac{1}{2}\alpha(X)\alpha(Y)$, $(e_1, e_2, e_3, e_4)$ is a local orthonormal frame on $M$.

At the end of this section, we will show some basic geometric structures defined by the twistor method.

L. Berard Bergery, T. Ochiai\cite{BO}, and N. R. O'Brian, J. H. Rawnsley\cite{OR} proposed a general construction of the twistor space $Z$ for an arbitrary even dimensional manifold with $G$-structure. They also studied the integrability of a natural almost complex structure $\mathbb{J}_Z$ (similar to $\mathbb{J}_+$) on $Z$. In particular, for a Hermitian manifold $(M, J, h)$, N.R. O'Brian and J.H. Rawnsley\cite{OR} showed that the Grassmann bundles $G_k(TM)$ of complex $k$-planes (i.e. real $J$-stable $2k$-planes) in $TM$ can be considered as the reduced twistor space of $M$. In fact, $G_k(TM)$ can be identified with the complex Grassmann bundle of the
complex tangent bundle $T^{1,0}M$. Set $Z=G_k(T^{1,0}M)$. If $(M, J, h)$ is a Hermitian surface, then $Z=\mathbb{P}(T^{1,0}M)$ coincides with the following associated bundle definition of the twistor space of 4-manifold\cite{Va2},
$$Z=O_-(M)\times_{SO(4)}SO(4)/U(2),$$
where $O_-(M)$ is the $SO(4)$-principal bundle of all negative orthonormal frames over $M$.

Unlike the Riemannian twistor geometry, there are more natural geometry structures on $Z=\mathbb{P}(T^{1,0}M)$ induced by various Hermitian connections.
In the following, we will use the language of principal bundle to study the geometry structures on $Z=\mathbb{P}(T^{1,0}M)$ as in\cite{Va2}. We have the following commutative diagram:
$$\xymatrix{
  Q=U(M) \ar[d]_{\pi_1} \ar[dr]^{\pi}        \\
  Z \ar[r]_{\pi_2}  & M              }
  $$
where $\pi$ and $\pi_2$ are the standard projections, $\pi_1((u_1,u_2))=Span_\mathbb{C}\{u_1\}$.

Associated with any connection $\phi=(\phi_b^a)$ on the principal bundle $Q$, there exist four distinguished almost Hermitian structures (up to conjugation) on $Z=\mathbb{P}(T^{1,0}M)$. As in\cite{JR,Va2,FZ}, to describe these almost Hermitian structures, we first define locally $(1,0)$-forms
on $Z=\mathbb{P}(T^{1,0}M)$. On $Q$, together with the connection form $\phi=(\phi_b^a)$, there exists a canonical $\mathbb{C}^2$-valued 1-form
$\varphi=(\varphi^1, \varphi^2)^t$. If $u: U\subset Z\rightarrow Q$ is a local section of the fibration $\pi_1: Q\rightarrow Z$,
then $\{u^\ast\varphi^1, u^\ast\varphi^2, u^\ast\overline{\varphi^1}, u^\ast\overline{\varphi^2}, u^\ast\phi_2^1, u^\ast\overline{\phi_2^1}\}$ is a local basis
of complex cotangent bundle of $Z=\mathbb{P}(T^{1,0}M)$. Hereafter, for convenience, we omit the pull-back mapping $u^\ast$ without ambiguity.

Now, set $\varphi^3=\phi_2^1$, we can define four natural almost complex structures $\mathbb{J}_i$ on $Z=\mathbb{P}(T^{1,0}M)$ as follows.

~~~~~~~~~~~~~~~~~~~~~~~~~~~~~~~$\mathbb{J}_1$:~~a basis of $(1, 0)$-forms is $\{\varphi^1, \overline{\varphi^2}, \varphi^3\}$;

~~~~~~~~~~~~~~~~~~~~~~~~~~~~~~~$\mathbb{J}_2$:~~a basis of $(1, 0)$-forms is $\{\varphi^1, \overline{\varphi^2}, \overline{\varphi^3}\}$;

~~~~~~~~~~~~~~~~~~~~~~~~~~~~~~~$\mathbb{J}_3$:~~a basis of $(1, 0)$-forms is $\{\varphi^1, \varphi^2, \varphi^3\}$;

~~~~~~~~~~~~~~~~~~~~~~~~~~~~~~~$\mathbb{J}_4$:~~a basis of $(1, 0)$-forms is $\{\varphi^1, \varphi^2, \overline{\varphi^3}\}$.

\noindent Of course, the above constructions are well-defined. Indeed, if $\hat{u}: V\subset Z\rightarrow Q$ is another local section of the fibration $\pi_1: Q\rightarrow Z$, then $\hat{u}=u\cdot \mathfrak{a}^{-1}$, where $\mathfrak{a}\in C^\infty(U\cap V, U(1)\times U(1))$, i.e. $\mathfrak{a}$=diag($e^{i\beta_1}$, $e^{i\beta_2}$). From\cite{KN}, by direct calculations, we have
$\hat{u}^\ast \varphi^1=e^{i\beta_1} u^\ast \varphi^1$, $\hat{u}^\ast \varphi^2=e^{i\beta_2} u^\ast \varphi^2$,
$\hat{u}^\ast \varphi^3=e^{i(\beta_1-\beta_2)} u^\ast \varphi^3$. Moreover, the pull-back forms of $\varphi^1\wedge\overline{\varphi^1}$, $\varphi^2\wedge\overline{\varphi^2}$, $\varphi^3\wedge\overline{\varphi^3}$ and $\varphi^1\wedge\overline{\varphi^2}\wedge\overline{\varphi^3}$ are globally defined on $Z$. Hence, the first Chern class $c_1(Z, \mathbb{J}_2)=0$\cite{ES,Va2}.

As in the Riemannian twistor space\cite{JR, FZ}, there exists a natural family of Riemannian metrics $h_\lambda$ on $Z=\mathbb{P}(T^{1,0}M)$,
\begin{equation}
h_\lambda=u^\ast(\varphi^1\cdot\overline{\varphi^1}+\varphi^2\cdot\overline{\varphi^2}+\lambda^2\varphi^3\cdot\overline{\varphi^3}),
\end{equation}
where parameter $\lambda>0$, $u$ is a local section of the fibration $\pi_1: Q\rightarrow Z$.

From the definition of $\mathbb{J}_i$, it is easy to see that $\mathbb{J}_i$ are orthogonal almost complex structures with respect to
 $h_\lambda$. The associated fundamental 2-forms are denoted by
 \begin{equation}
\mathbb{K}_1(\lambda)=\sqrt{-1}(\varphi^1\wedge\overline{\varphi^1}+\overline{\varphi^2}\wedge \varphi^2+\lambda^2\varphi^3\wedge\overline{\varphi^3}),
\end{equation}
\begin{equation}
\mathbb{K}_2(\lambda)=\sqrt{-1}(\varphi^1\wedge\overline{\varphi^1}+\overline{\varphi^2}\wedge \varphi^2+\lambda^2\overline{\varphi^3}\wedge \varphi^3),
\end{equation}
\begin{equation}
\mathbb{K}_3(\lambda)=\sqrt{-1}(\varphi^1\wedge\overline{\varphi^1}+\varphi^2\wedge \overline{\varphi^2}+\lambda^2\varphi^3\wedge\overline{\varphi^3}),
\end{equation}
\begin{equation}
\mathbb{K}_4(\lambda)=\sqrt{-1}(\varphi^1\wedge\overline{\varphi^1}+\varphi^2\wedge \overline{\varphi^2}+\lambda^2\overline{\varphi^3}\wedge\varphi^3).
\end{equation}
Here we omit the pull-back mapping $u^\ast$. In fact, we can define 3-parameter family of almost Hermitian structures
$h_{\lambda_1,\lambda_2,\lambda_3}$ on $Z=\mathbb{P}(T^{1,0}M)$ as follows:
\begin{equation}
h_{\lambda_1,\lambda_2,\lambda_3}=\lambda_1^2\varphi^1\cdot\overline{\varphi^1}+\lambda_2^2\varphi^2\cdot\overline{\varphi^2}
+\lambda_3^2\varphi^3\cdot\overline{\varphi^3},
\end{equation}
where parameters $\lambda_1>0$, $\lambda_2>0$,$\lambda_3>0$. In the present paper, we only consider the natural case $h_\lambda$.

In the following sections, we consider some canonical connections on the principal bundle $Q$, and then study the corresponding
almost Hermitian twistor geometry by using the method of moving frames.

\section{Twistor geometry I: induced by the Lichnerowicz connection}

In this section, corresponding to the Lichnerowicz connection (also called the first canonical connection) on the principal bundle of unitary frames over a Hermitian surface $M$, we define four natural almost Hermitian
structures, denoted by $(\mathbb{K}_i^L(\lambda), \mathbb{J}_i^L)$, for $i=1, 2, 3, 4,$ on the twistor space $Z=\mathbb{P}(T^{1,0}M)$. Then
we consider the integrability and conformal property of $\mathbb{J}_i^L$.  Some metric properties of $(\mathbb{K}_i^L(\lambda), \mathbb{J}_i^L)$ on $Z=\mathbb{P}(T^{1,0}M)$ are also studied.

Let $(M, J, h)$ be a Hermitian surface with the natural orientation $\frac{F^2}{2}$. Let $P=SO(M)$ be the $SO(4)$-principal bundle of oriented orthonormal frames over $M$, $Q=U(M)$  the $U(2)$-principal bundle of unitary frames over $M$. Of course, $Q$ is a principal subbundle  of $P$.
The $\mathbb{R}^4$-valued canonical form on $P$, denoted by $\theta=(\theta^1, \theta^2, \theta^3, \theta^4)^t$, is
given by
$$\theta(X)=e^{-1}({\pi_P}_{\ast}(X)),~~~~X\in T_{(x,e)}P,$$
where $e$ is identified with a linear map $e:\mathbb{R}^4\rightarrow T_{x}M$, $\pi_P:P\rightarrow M$ is the projection.

The $\mathfrak{so}(4)$-valued Levi-Civita connection form and curvature form are denoted by
$\omega=(\omega_j^i)$ and $\Omega=(\Omega_j^i)$, respectively. Then the structure equations of $M$ are\cite{KN}
\begin{eqnarray}
&&d\theta^i=-\omega_j^i\wedge\theta^j,\\
&&d\omega_j^i=-\omega_k^i\wedge\omega_j^k+\Omega_j^i,
\end{eqnarray}
where $\Omega_j^i=\frac{1}{2}R_{ijkl}\theta^k\wedge\theta^l$,
$R_{ijkl}$ are functions on $P$ defining the Riemannian curvature tensor of $M$.

The canonical form on $Q$ is also denoted by $\theta$, which is the restriction of the canonical form of $P$ to $Q$. The Lie algebra decomposition $\mathfrak{so}(4)=\mathfrak{u}(2)\oplus\mathfrak{m}$ induces a decomposition of $\omega$:
$$\omega|_Q=\omega_{\mathfrak{u}(2)}\oplus\omega_\mathfrak{m}.$$
Set $\phi=\omega_{\mathfrak{u}(2)}$, in fact, $\phi$ defines a connection on the principal bundle $Q$, which is the Lichnerowicz connection.

Set
$$\varphi^1=\frac{1}{\sqrt{2}}(\theta^1+\sqrt{-1}\theta^2), \varphi^2=\frac{1}{\sqrt{2}}(\theta^3+\sqrt{-1}\theta^4).$$
Then the structure equations on $Q$ for the connection $\phi=(\phi_b^a)$ are
\begin{equation}
d\varphi^1=-\phi_1^1\wedge\varphi^1-\phi_2^1\wedge\varphi^2+\tau^1,
\end{equation}
\begin{equation}
d\varphi^2=-\phi_1^2\wedge\varphi^1-\phi_2^2\wedge\varphi^2+\tau^2,
\end{equation}
where
\begin{eqnarray*}
&&\phi_1^1=-\sqrt{-1}\omega_2^1,\\
&&\phi_2^2=-\sqrt{-1}\omega_4^3,\\
&&\phi_2^1=\frac{1}{2}[(\omega_3^1+\omega_4^2)+\sqrt{-1}(\omega_3^2-\omega_4^1)],\\
&&\tau^1=\frac{1}{2}[(\omega_4^2-\omega_3^1)-\sqrt{-1}(\omega_3^2+\omega_4^1)]\wedge \overline{\varphi^2},\\
&&\tau^2=\frac{1}{2}[(\omega_3^1-\omega_4^2)+\sqrt{-1}(\omega_3^2+\omega_4^1)]\wedge \overline{\varphi^1}.
\end{eqnarray*}
Set $\mu=\frac{1}{2}[(\omega_4^2-\omega_3^1)-\sqrt{-1}(\omega_3^2+\omega_4^1)]$. Then $\tau^1=\mu\wedge \overline{\varphi^2}$, $\tau^2=-\mu\wedge \overline{\varphi^1}$.

As in section 2, set $\varphi^3=\phi_2^1$, we can define four natural almost complex structures, denoted by $\mathbb{J}_i^L$, on $Z=\mathbb{P}(T^{1,0}M)$ as follows.

~~~~~~~~~~~~~~~~~~~~~~~~~~~~~~~$\mathbb{J}_1^L$:~~a basis of $(1, 0)$-forms is $\{\varphi^1, \overline{\varphi^2}, \varphi^3\}$;

~~~~~~~~~~~~~~~~~~~~~~~~~~~~~~~$\mathbb{J}_2^L$:~~a basis of $(1, 0)$-forms is $\{\varphi^1, \overline{\varphi^2}, \overline{\varphi^3}\}$;

~~~~~~~~~~~~~~~~~~~~~~~~~~~~~~~$\mathbb{J}_3^L$:~~a basis of $(1, 0)$-forms is $\{\varphi^1, \varphi^2, \varphi^3\}$;

~~~~~~~~~~~~~~~~~~~~~~~~~~~~~~~$\mathbb{J}_4^L$:~~a basis of $(1, 0)$-forms is $\{\varphi^1, \varphi^2, \overline{\varphi^3}\}$.

\noindent A natural family of $\mathbb{J}_i^L$-compatible Riemannian metrics, denoted by $h_\lambda^L$, on $Z=\mathbb{P}(T^{1,0}M)$ is
\begin{equation}
h_\lambda^L=u^\ast(\varphi^1\cdot\overline{\varphi^1}+\varphi^2\cdot\overline{\varphi^2}+\lambda^2\varphi^3\cdot\overline{\varphi^3}),
\end{equation}
where parameter $\lambda>0$, $u$ is a local section of the fibration $\pi_1: Q\rightarrow Z$. The associated fundamental 2-forms are denoted by $\mathbb{K}_i^L(\lambda)$, with respect to $\mathbb{J}_i^L$.

From the structure equation (3.2), we have
\begin{equation}
d\varphi^3=\sqrt{-1}(\omega_2^1-\omega_4^3)\wedge \varphi^3+\frac{1}{2}[(\Omega_3^1+\Omega_4^2)+\sqrt{-1}(\Omega_3^2-\Omega_4^1)],
\end{equation}
where
\begin{align}\label{E:1}
\frac{1}{2}[(\Omega_3^1+\Omega_4^2)+\sqrt{-1}(\Omega_3^2-\Omega_4^1)]
                      &=\mathbf{R}_{\bar{1}212}\varphi^1\wedge\varphi^2+\mathbf{R}_{\bar{1}2\bar{1}\bar{2}}\overline{\varphi^1}\wedge\overline{\varphi^2}
                      +\mathbf{R}_{\bar{1}21\bar{1}}\varphi^1\wedge\overline{\varphi^1}\notag\\
                      &~~~~+\mathbf{R}_{\bar{1}22\bar{2}}\varphi^2\wedge\overline{\varphi^2}+\mathbf{R}_{\bar{1}21\bar{2}}\varphi^1\wedge\overline{\varphi^2}
                      +\mathbf{R}_{\bar{1}2\bar{1}2}\overline{\varphi^1}\wedge\varphi^2,
\end{align}
where, for example
\begin{eqnarray*}
&&\mathbf{R}_{\bar{1}212}=R(\overline{u_1}, u_2, u_1, u_2)\\
&&~~~~~~~~~=\frac{1}{4}R(e_1+\sqrt{-1}e_2,e_3-\sqrt{-1}e_4,e_1-\sqrt{-1}e_2,e_3-\sqrt{-1}e_4)\\
&&~~~~~~~~~=\frac{1}{4}[(R_{1313}-R_{2424}+R_{2323}-R_{1414})-2\sqrt{-1}(R_{1314}+R_{2324})],
\end{eqnarray*}
where $(u_1, u_2)=\frac{1}{\sqrt{2}}(e_1-\sqrt{-1}e_2,e_3-\sqrt{-1}e_4)\in Q$.  Others are similar.

\begin{prop}
Let $(M, J, h)$ be a Hermitian surface. For the Lichnerowicz connection  on $Q=U(M)$, the almost complex structures $\mathbb{J}_1^L$, $\mathbb{J}_2^L$, $\mathbb{J}_3^L$ and $\mathbb{J}_4^L$ on the twistor space $Z=\mathbb{P}(T^{1,0}M)$ have the following properties:
\begin{enumerate}[\upshape (i)]
\item $\mathbb{J}_1^L$ is integrable if and only if $(M, J, h)$ is self-dual;
\item $\mathbb{J}_2^L$ is not integrable;
\item $\mathbb{J}_3^L$ (or $\mathbb{J}_4^L$) is integrable if and only if the Ricci tensor of the Levi-Civita connection on  $(M, J, h)$ is $J$-invariant.
\end{enumerate}
\end{prop}

\begin{proof}
It is well-known that $\mathbb{J}_i^L$ is integrable if and only if 
$$d\Omega_{\mathbb{J}_i^L}^{1,0}(Z)\subseteq \Omega_{\mathbb{J}_i^L}^{2,0}(Z)\oplus\Omega_{\mathbb{J}_i^L}^{1,1}(Z).$$ 
We check this by using the structure equations (3.3), (3.4) and (3.6).

(i) Together with (3.7), we can see that $\mathbb{J}_1^L$ is integrable if and only if $\mathbf{R}_{\bar{1}2\bar{1}2}=0$ on $Q$. Now, we claim that
\begin{equation}
\mathbf{R}_{\bar{1}2\bar{1}2}=0
\end{equation}
on $Q$ if and only if $(M, J, h)$ is self-dual.

Fix an unitary frame $(u_1, u_2)\in Q$. Then any unitary frame $(\hat{u}_1, \hat{u}_2)$ can be written as
$$(\hat{u}_1, \hat{u}_2)=(u_1, u_2)\mathfrak{a},$$
where $\mathfrak{a}=\left(\begin{array}{cc}a_1&a_2\\
a_3&a_4\end{array}\right) \in U(2)$. If $R(\overline{\hat{u}_1},\hat{u}_2, \overline{\hat{u}_1}, \hat{u}_2)=0$, by direct calculations, we have
\begin{eqnarray*}
&&(a_1\overline{a_2})^2(\mathbf{R}_{1\bar{1}1\bar{1}}+\mathbf{R}_{2\bar{2}2\bar{2}}-2\mathbf{R}_{1\bar{1}2\bar{2}}+2\mathbf{R}_{1\bar{2}\bar{1}2})\\
&&+2a_1^2\overline{a_2}\overline{a_4}(\mathbf{R}_{1\bar{1}1\bar{2}}-\mathbf{R}_{1\bar{2}2\bar{2}})
-2\overline{a_2}^2a_1a_3(\mathbf{R}_{1\bar{1}\bar{1}2}-\mathbf{R}_{\bar{1}22\bar{2}})=0.
\end{eqnarray*}
Thus, if $\mathbf{R}_{\bar{1}2\bar{1}2}=0$ on $Q$, then we also have 
\begin{equation}
\mathbf{R}_{1\bar{1}1\bar{2}}-\mathbf{R}_{1\bar{2}2\bar{2}}=0,
\end{equation}
\begin{equation}
\mathbf{R}_{1\bar{1}1\bar{1}}+\mathbf{R}_{2\bar{2}2\bar{2}}-2\mathbf{R}_{1\bar{1}2\bar{2}}+2\mathbf{R}_{1\bar{2}\bar{1}2}=0,
\end{equation}
on $Q$.

By the definition of anti-self-dual Weyl curvature operator $W^-$, it follows that
the equations (3.8), (3.9) and (3.10) on $Q$ are equivalent to $W^-=0$, i.e. $(M, J, h)$ is self-dual.

(ii) From the structure equations of the Lichnerowicz connection on $Q$, it is obvious that $\mathbb{J}_2^L$ is not integrable.

(iii) Since $(M,J)$ is a complex surface, then $\mathbb{J}_3^L$ is integrable if and only if $\mathbf{R}_{\bar{1}2\bar{1}\bar{2}}=0$ on $Q$; $\mathbb{J}_4^L$ is integrable if and only if $\mathbf{R}_{\bar{1}212}=0$ on $Q$. In fact, $\mathbf{R}_{\bar{1}212}=0$ on $Q$ is equivalent to
$\mathbf{R}_{\bar{1}2\bar{1}\bar{2}}=0$ on $Q$.

Now, we write the integrability condition $\mathbf{R}_{\bar{1}2\bar{1}\bar{2}}=0$ in a real version. For
\begin{eqnarray*}
&&\mathbf{R}_{\bar{1}2\bar{1}\bar{2}}=R(\overline{u_1}, u_2, \overline{u_1}, \overline{u_2})\\
&&~~~~~~~~~=\frac{1}{4}R(e_1+\sqrt{-1}e_2,e_3-\sqrt{-1}e_4,e_1+\sqrt{-1}e_2,e_3+\sqrt{-1}e_4)\\
&&~~~~~~~~~=\frac{1}{4}[(R_{1313}-R_{2424}+R_{1414}-R_{2323})+2\sqrt{-1}(R_{1323}+R_{1424})],
\end{eqnarray*}
it follows that $\mathbf{R}_{\bar{1}2\bar{1}\bar{2}}=0$ is equivalent to
$$R_{1313}-R_{2424}+R_{1414}-R_{2323}=0,~~R_{1323}+R_{1424}=0,$$
i.e. the Ricci tensor of the Levi-Civita connection on  $(M, J, h)$ satisfies
$$Ric(e_1,e_1)=Ric(e_2,e_2),~~Ric(e_1,e_2)=0.$$
Therefore, the integrability condition $\mathbf{R}_{\bar{1}2\bar{1}\bar{2}}=0$ on $Q$ is equivalent to the Ricci tensor of the Levi-Civita connection on  $(M, J, h)$ is $J$-invariant.

\end{proof}

\begin{rem} In fact, from the structure equations of any Hermitian connection on  $Q=U(M)$, the second almost complex structure $\mathbb{J}_2$ is never integrable. I. Vaisman\cite{Va2} showed that the almost complex structures $\mathbb{J}_3$ and $\mathbb{J}_4$ are integrable or not simultaneously, for any Hermitian  connection. The condition (iii) in the above proposition appeared in the Riemannian Goldberg-Sachs theorem, which was studied by V. Apostolov and P. Gauduchon\cite{AG}. It is easy to see that the condition (iii) is satisfied for a Hermitian Einstein surface or a K\"{a}hler surface.
\end{rem}

It is a natural question to see how the almost complex structures $\mathbb{J}_i^L$ on $Z$ depend on the choice of metric $h$ in the conformal class $[h]$. For
a Hermitian surface $(M, J, \tilde{h}=e^{2f}h)$ with a conformally related metric $\tilde{h}$, let $\tilde{\pi}: \tilde{Q}\rightarrow M$  be the $U(2)$-principal
bundle of $\tilde{h}$-unitary frames on $M$. The corresponding canonical form,  Lichnerowicz connection form and torsion form  on $\tilde{Q}$ are denoted by $\tilde{\varphi}=(\tilde{\varphi}^1, \tilde{\varphi}^2)^t$, $\tilde{\phi}=(\tilde{\phi}_b^a)$, $\tilde{\tau}=(\tilde{\tau}^1, \tilde{\tau}^2)^t$, respectively. Then
$$d\tilde{\varphi}^a=-\tilde{\phi}_b^a\wedge\tilde{\varphi}^b+\tilde{\tau}^a.$$
We have a bundle isomorphism
$$\mathfrak{f}: \tilde{Q}\rightarrow Q,~~~~\mathfrak{f}(x, \tilde{u}):=(x, e^f\tilde{u}).$$
Thus $\mathfrak{f}^\ast\varphi=e^{-f}\tilde{\varphi}$. Set $df=f_b\tilde{\varphi}^b+\overline{f_b\tilde{\varphi}^b}$. Then by the structure equations of $\phi$ and $\tilde{\phi}$, we obtain
\begin{equation}
\mathfrak{f}^\ast\phi_b^a=\tilde{\phi}_b^a+\overline{f_a\tilde{\varphi}^b}-f_b\tilde{\varphi}^a.
\end{equation}
In particular, $\mathfrak{f}^\ast\phi_2^1=\tilde{\phi}_2^1+\overline{f_1\tilde{\varphi}^2}-f_2\tilde{\varphi}^1$. From the constructions of almost complex structures $\mathbb{J}_i^L$ on $Z$, and the following commutative diagram,
$$\xymatrix{
  \tilde{Q} \ar[rr]^{\mathfrak{f}} \ar[dr]_{\tilde{\pi}_1}
                &  &    Q\ar[dl]^{\pi_1}    \\
                & Z                },
$$
we obtain
\begin{prop}
$\mathbb{J}_1^L$ is conformally invariant; $\mathbb{J}_2^L$, $\mathbb{J}_3^L$ and $\mathbb{J}_4^L$ are conformally invariant only under change of scale.
\end{prop}

Next, we study some metric properties of $(Z=\mathbb{P}(T^{1,0}M), \mathbb{J}_i^L, \mathbb{K}_i^L(\lambda))$.

\begin{thm} Let $(M, J, h)$ be a Hermitian surface. For the Lichnerowicz connection  on $Q=U(M)$, the almost Hermitian structures $(\mathbb{J}_i^L, \mathbb{K}_i^L(\lambda))$ on the twistor space $Z=\mathbb{P}(T^{1,0}M)$ have the following properties:
\begin{enumerate}[\upshape (i)]
\item $d\mathbb{K}_1^L(\lambda)=0$ if and only if $(M, J, h)$ is self-dual and Einstein with scalar curvature $s=\frac{24}{\lambda^2}$;
\item $d\mathbb{K}_2^L(\lambda)=0$ if and only if $(M, J, h)$ is self-dual and Einstein with scalar curvature $s=-\frac{24}{\lambda^2}$;
\item $d\mathbb{K}_3^L(\lambda)=0$ (or equivalently, $d\mathbb{K}_4^L(\lambda)=0$) if and only if $(M, J, h)$ is a flat K\"{a}hler surface.
\end{enumerate}
\end{thm}

\begin{proof}
(i) For the exterior differentiation of $\mathbb{K}_1^L(\lambda)$, we have
\begin{align}\label{E:1}
d\mathbb{K}_1^L(\lambda)
                      &=\sqrt{-1}\{2\overline{\varphi^1}\wedge\varphi^2\wedge\varphi^3-2\varphi^1\wedge\overline{\varphi^2}\wedge\overline{\varphi^3}\notag\\
                      &~~~~~+\frac{1}{2}\lambda^2[(\Omega_3^1+\Omega_4^2)+\sqrt{-1}(\Omega_3^2-\Omega_4^1)]\wedge\overline{\varphi^3}\notag\\
                      &~~~~~-\frac{1}{2}\lambda^2[(\Omega_3^1+\Omega_4^2)-\sqrt{-1}(\Omega_3^2-\Omega_4^1)]\wedge\varphi^3\}.
\end{align}
Thus $d\mathbb{K}_1^L(\lambda)=0$ if and only if
$$2\varphi^1\wedge\overline{\varphi^2}\wedge\overline{\varphi^3}
=\frac{1}{2}\lambda^2[(\Omega_3^1+\Omega_4^2)+\sqrt{-1}(\Omega_3^2-\Omega_4^1)]\wedge\overline{\varphi^3}.$$
Together with (3.7), it follows that $d\mathbb{K}_1^L(\lambda)=0$ if and only if
\begin{equation}
\lambda^2\mathbf{R}_{\bar{1}21\bar{2}}=2,~~\mathbf{R}_{\bar{1}2\bar{1}2}=0,
\end{equation}
\begin{equation}
\mathbf{R}_{\bar{1}21\bar{1}}=\mathbf{R}_{\bar{1}212}=0~~(\textrm{equivalent to}~~\mathbf{R}_{\bar{1}22\bar{2}}=\mathbf{R}_{\bar{1}2\bar{1}\bar{2}}=0),
\end{equation}
on $Q$.

From the proof of Theorem 3.1, we have obtained that $\mathbf{R}_{\bar{1}2\bar{1}2}=0$ on $Q$ if and only if $(M, J, h)$ is self-dual. If $(M, J, h)$ is self-dual, then $h(\hat{R}(\epsilon_2^-),\epsilon_2^-)=\mathbf{R}_{\bar{1}21\bar{2}}=\frac{s}{12}$, where $s$ is the scalar curvature of $(M, J, h)$. Thus
$\lambda^2\mathbf{R}_{\bar{1}21\bar{2}}=2$ implies $s=\frac{24}{\lambda^2}$.

If $(u_1,u_2)\in Q$, then $\frac{1}{\sqrt{2}}(u_1+u_2, u_1-u_2)\in Q$. Meanwhile, (3.14) implies
\begin{equation}
\mathbf{R}_{\bar{1}11\bar{1}}=\mathbf{R}_{\bar{2}22\bar{2}},~~\mathbf{R}_{121\bar{1}}=\mathbf{R}_{122\bar{2}}
\end{equation}
on $Q$.
From (3.14) and (3.15), we have
$h(\hat{R}(\epsilon_l^+),\epsilon_m^-)=0$, $l,m=1,2,3$. Hence $(M, J, h)$ is Einstein.

Conversely, it is easy to prove by direct calculations.

(ii)  For the exterior differentiation of $\mathbb{K}_2^L(\lambda)$, we have
\begin{align}\label{E:1}
d\mathbb{K}_2^L(\lambda)
                      &=\sqrt{-1}\{2\overline{\varphi^1}\wedge\varphi^2\wedge\varphi^3-2\varphi^1\wedge\overline{\varphi^2}\wedge\overline{\varphi^3}\notag\\
                      &~~~~~-\frac{1}{2}\lambda^2[(\Omega_3^1+\Omega_4^2)+\sqrt{-1}(\Omega_3^2-\Omega_4^1)]\wedge\overline{\varphi^3}\notag\\
                      &~~~~~+\frac{1}{2}\lambda^2[(\Omega_3^1+\Omega_4^2)-\sqrt{-1}(\Omega_3^2-\Omega_4^1)]\wedge\varphi^3\}.
\end{align}
Thus $d\mathbb{K}_2^L(\lambda)=0$ if and only if
$$2\varphi^1\wedge\overline{\varphi^2}\wedge\overline{\varphi^3}
+\frac{1}{2}\lambda^2[(\Omega_3^1+\Omega_4^2)+\sqrt{-1}(\Omega_3^2-\Omega_4^1)]\wedge\overline{\varphi^3}=0.$$
Together with (3.7), it follows that $d\mathbb{K}_2^L(\lambda)=0$ if and only if
\begin{equation}
2+\lambda^2\mathbf{R}_{\bar{1}21\bar{2}}=0,~~\mathbf{R}_{\bar{1}2\bar{1}2}=0,
\end{equation}
\begin{equation}
\mathbf{R}_{\bar{1}21\bar{1}}=\mathbf{R}_{\bar{1}212}=0~~(\textrm{equivalent to}~~\mathbf{R}_{\bar{1}22\bar{2}}=\mathbf{R}_{\bar{1}2\bar{1}\bar{2}}=0),
\end{equation}
on  $Q$.

Therefore, the results follow from the proof as in (i).

(iii) For the exterior differentiation of $\mathbb{K}_3^L(\lambda)$, we have
\begin{align}\label{E:1}
d\mathbb{K}_3^L(\lambda)
                      &=\sqrt{-1}\{2\bar{\mu}\wedge\varphi^1\wedge\varphi^2-2\mu\wedge\overline{\varphi^1}\wedge\overline{\varphi^2}\notag\\
                      &~~~~~+\frac{1}{2}\lambda^2[(\Omega_3^1+\Omega_4^2)+\sqrt{-1}(\Omega_3^2-\Omega_4^1)]\wedge\overline{\varphi^3}\notag\\
                      &~~~~~-\frac{1}{2}\lambda^2[(\Omega_3^1+\Omega_4^2)-\sqrt{-1}(\Omega_3^2-\Omega_4^1)]\wedge\varphi^3\}.
\end{align}
Thus $d\mathbb{K}_3^L(\lambda)=0$ if and only if
$$\mu\wedge\overline{\varphi^1}\wedge\overline{\varphi^2}=0,~~\Omega_3^1+\Omega_4^2=\Omega_3^2-\Omega_4^1=0.$$

For the exterior differentiation of $\mathbb{K}_4^L(\lambda)$, we have
\begin{align}\label{E:1}
d\mathbb{K}_4^L(\lambda)
                      &=\sqrt{-1}\{2\bar{\mu}\wedge\varphi^1\wedge\varphi^2-2\mu\wedge\overline{\varphi^1}\wedge\overline{\varphi^2}\notag\\
                      &~~~~~-\frac{1}{2}\lambda^2[(\Omega_3^1+\Omega_4^2)+\sqrt{-1}(\Omega_3^2-\Omega_4^1)]\wedge\overline{\varphi^3}\notag\\
                      &~~~~~+\frac{1}{2}\lambda^2[(\Omega_3^1+\Omega_4^2)-\sqrt{-1}(\Omega_3^2-\Omega_4^1)]\wedge\varphi^3\}.
\end{align}
Thus $d\mathbb{K}_4^L(\lambda)=0$ if and only if
$$\mu\wedge\overline{\varphi^1}\wedge\overline{\varphi^2}=0,~~\Omega_3^1+\Omega_4^2=\Omega_3^2-\Omega_4^1=0.$$

It follows that the symplectic conditions for the metrics $\mathbb{K}_3^L(\lambda)$ and $\mathbb{K}_4^L(\lambda)$ are the same. Moreover, since $(M, J)$ is a complex surface, then $\mu\wedge\overline{\varphi^1}\wedge\overline{\varphi^2}=0$ if and only if $\mu=0$, i.e. $(M, J, h)$ is a  K\"{a}hler surface. It is well-known that for a K\"{a}hler surface, the curvature forms on $Q$ satisfy $\Omega_3^1=\Omega_4^2$, $\Omega_4^1=-\Omega_3^2$. Now, we claim that $\Omega_3^1=\Omega_4^2=\Omega_4^1=\Omega_3^2=0$ on $Q$ if and only if $(M, J, h)$ is flat.

Given a $J$-adapted orthonormal frame $(e_1, e_2=Je_1, e_3, e_4=Je_3)$, we define a new $J$-adapted orthonormal frame $\frac{1}{\sqrt{2}}(e_1+e_3, Je_1+Je_3,e_1-e_3, Je_1-Je_3)$. So $\Omega_4^1=\Omega_3^2=0$  on $Q$ implies
$h(R(\cdot,\cdot)(Je_1-Je_3),e_1+e_3)=0$, and then $\Omega_2^1=\Omega_4^3$ on $Q$. From $\Omega_3^1=\Omega_3^2=0$ and $\Omega_2^1=\Omega_4^3$, it follows $R_{1212}=R_{3412}=-R_{3124}-R_{3241}=0$. Therefore, $\Omega_2^1=\Omega_4^3=0$ on $Q$. Combining these facts, $(M, J, h)$ must be flat.

\end{proof}

\begin{thm} Let $(M, J, h)$ be a Hermitian surface. For the Lichnerowicz connection  on $Q=U(M)$, the almost Hermitian structures $(\mathbb{J}_i^L, \mathbb{K}_i^L(\lambda))$ on the twistor space $Z=\mathbb{P}(T^{1,0}M)$ have the following properties:
\begin{enumerate}[\upshape (i)]
\item $d\mathbb{K}_1^L(\lambda)\wedge\mathbb{K}_1^L(\lambda)=0$ (or equivalently, $d\mathbb{K}_2^L(\lambda)\wedge\mathbb{K}_2^L(\lambda)=0$) if and only if $(M, J, h)$ is self-dual;
\item $d\mathbb{K}_3^L(\lambda)\wedge\mathbb{K}_3^L(\lambda)=0$ (or equivalently, $d\mathbb{K}_4^L(\lambda)\wedge\mathbb{K}_4^L(\lambda)=0$) if and only if $(M, J, h)$ is a K\"{a}hler Einstein surface.
\end{enumerate}
\end{thm}

\begin{proof}
(i) From the exterior differential formula (3.12), we have
\begin{align}\label{E:1}
\mathbb{K}_1^L(\lambda)\wedge d\mathbb{K}_1^L(\lambda)
                      &=\lambda^2[(\mathbf{R}_{\bar{1}22\bar{2}}-\mathbf{R}_{\bar{1}21\bar{1}})
                      \varphi^1\wedge\overline{\varphi^1}\wedge\overline{\varphi^2}\wedge\varphi^2\wedge \overline{\varphi^3}\notag\\
                      &~~~~+(\overline{\mathbf{R}_{\bar{1}22\bar{2}}}-\overline{\mathbf{R}_{\bar{1}21\bar{1}}})
                      \varphi^1\wedge\overline{\varphi^1}\wedge\overline{\varphi^2}\wedge\varphi^2\wedge \varphi^3].
\end{align}

From the exterior differential formula (3.16), we have
\begin{align}\label{E:1}
\mathbb{K}_2^L(\lambda)\wedge d\mathbb{K}_2^L(\lambda)
                      &=\lambda^2[(\mathbf{R}_{\bar{1}21\bar{1}}-\mathbf{R}_{\bar{1}22\bar{2}})
                      \varphi^1\wedge\overline{\varphi^1}\wedge\overline{\varphi^2}\wedge\varphi^2\wedge \overline{\varphi^3}\notag\\
                      &~~~~+(\overline{\mathbf{R}_{\bar{1}21\bar{1}}}-\overline{\mathbf{R}_{\bar{1}22\bar{2}}})
                      \varphi^1\wedge\overline{\varphi^1}\wedge\overline{\varphi^2}\wedge\varphi^2\wedge \varphi^3].
\end{align}
Thus $d\mathbb{K}_1^L(\lambda)\wedge\mathbb{K}_1^L(\lambda)=0$ (or equivalently, $d\mathbb{K}_2^L(\lambda)\wedge\mathbb{K}_2^L(\lambda)=0$) if and only if
$\mathbf{R}_{\bar{1}21\bar{1}}-\mathbf{R}_{\bar{1}22\bar{2}}=0$ on $Q$. Now, we should prove that this condition is equivalent to $(M, J, h)$ is self-dual. The idea is the same as in the proof of Theorem 3.1 (i).

Fix an unitary frame $(u_1, u_2)\in Q$, then any unitary frame $(\hat{u}_1, \hat{u}_2)$ can be written as
$$(\hat{u}_1, \hat{u}_2)=(u_1, u_2)\mathfrak{a},$$
where $\mathfrak{a}=\left(\begin{array}{cc}a_1&a_2\\
a_3&a_4\end{array}\right) \in U(2)$. If $R(\overline{\hat{u}_1},\hat{u}_2, \hat{u}_1, \overline{\hat{u}_1})-R(\overline{\hat{u}_1},\hat{u}_2, \hat{u}_2, \overline{\hat{u}_2})=0$, by direct calculations, we have
\begin{eqnarray*}
&&2\overline{a_1}\overline{a_2}a_4^2\mathbf{R}_{\bar{1}2\bar{1}2}+2\overline{a_3}\overline{a_4}a_2^2\mathbf{R}_{1\bar{2}1\bar{2}}\\
&&+(2a_2\overline{a_2}-1)\overline{a_1}a_2(\mathbf{R}_{1\bar{1}1\bar{1}}+\mathbf{R}_{2\bar{2}2\bar{2}}
-2\mathbf{R}_{1\bar{1}2\bar{2}}+2\mathbf{R}_{1\bar{2}\bar{1}2})=0.
\end{eqnarray*}
Thus, if $\mathbf{R}_{\bar{1}21\bar{1}}-\mathbf{R}_{\bar{1}22\bar{2}}=0$ on $Q$, then the following equations hold:
\begin{equation}
\mathbf{R}_{\bar{1}2\bar{1}2}=0,~~
\mathbf{R}_{1\bar{1}1\bar{1}}+\mathbf{R}_{2\bar{2}2\bar{2}}-2\mathbf{R}_{1\bar{1}2\bar{2}}+2\mathbf{R}_{1\bar{2}\bar{1}2}=0.
\end{equation}

As we proved in Theorem 3.1, $\mathbf{R}_{\bar{1}2\bar{1}2}=0$ on $Q$  implies $(M, J, h)$ is self-dual. Conversely, it is easy to prove by direct calculations.

(ii) From the exterior differential formula (3.19), we have
\begin{align}\label{E:1}
\mathbb{K}_3^L(\lambda)\wedge d\mathbb{K}_3^L(\lambda)
                      &=-\lambda^2[(\mathbf{R}_{\bar{1}22\bar{2}}+\mathbf{R}_{\bar{1}21\bar{1}})
                      \varphi^1\wedge\overline{\varphi^1}\wedge\varphi^2\wedge\overline{\varphi^2}\wedge \overline{\varphi^3}\notag\\
                      &~~~~~+(\overline{\mathbf{R}_{\bar{1}22\bar{2}}}+\overline{\mathbf{R}_{\bar{1}21\bar{1}}})
                      \varphi^1\wedge\overline{\varphi^1}\wedge\varphi^2\wedge\overline{\varphi^2}\wedge \varphi^3\notag\\
                      &~~~~~+2(\bar{\mu}\wedge\varphi^1\wedge\varphi^2
                      -\mu\wedge\overline{\varphi^1}\wedge\overline{\varphi^2})\wedge \varphi^3\wedge\overline{\varphi^3}].
\end{align}

From the exterior differential formula (3.20), we have
\begin{align}\label{E:1}
\mathbb{K}_4^L(\lambda)\wedge d\mathbb{K}_4^L(\lambda)
                      &=\lambda^2[(\overline{\mathbf{R}_{\bar{1}22\bar{2}}}+\overline{\mathbf{R}_{\bar{1}21\bar{1}}})
                      \varphi^1\wedge\overline{\varphi^1}\wedge\varphi^2\wedge\overline{\varphi^2}\wedge \varphi^3\notag\\
                      &~~~~~+(\mathbf{R}_{\bar{1}22\bar{2}}+\mathbf{R}_{\bar{1}21\bar{1}})
                      \varphi^1\wedge\overline{\varphi^1}\wedge\varphi^2\wedge\overline{\varphi^2}\wedge \overline{\varphi^3}\notag\\
                      &~~~~~+2(\bar{\mu}\wedge\varphi^1\wedge\varphi^2
                      -\mu\wedge\overline{\varphi^1}\wedge\overline{\varphi^2})\wedge\varphi^3\wedge\overline{\varphi^3}].
\end{align}
Thus $d\mathbb{K}_3^L(\lambda)\wedge\mathbb{K}_3^L(\lambda)=0$ (or equivalently, $d\mathbb{K}_4^L(\lambda)\wedge\mathbb{K}_4^L(\lambda)=0$) if and only if
$\mu\wedge\overline{\varphi^1}\wedge\overline{\varphi^2}=0$ and $\mathbf{R}_{\bar{1}21\bar{1}}+\mathbf{R}_{\bar{1}22\bar{2}}=0$ on $Q$.

As proved in Theorem 3.4 (iii), $\mu\wedge\overline{\varphi^1}\wedge\overline{\varphi^2}=0$ if and only if $(M, J, h)$ is a K\"{a}hler surface.
Now, we write $\mathbf{R}_{\bar{1}21\bar{1}}$ and $\mathbf{R}_{\bar{1}22\bar{2}}$ in the following real version:
\begin{eqnarray*}
&&\mathbf{R}_{\bar{1}21\bar{1}}=R(\overline{u_1}, u_2, u_1, \overline{u_1})\\
&&~~~~~~~~~=\frac{1}{4}R(e_1+\sqrt{-1}e_2,e_3-\sqrt{-1}e_4,e_1-\sqrt{-1}e_2,e_1+\sqrt{-1}e_2)\\
&&~~~~~~~~~=\frac{1}{2}[(R_{1412}-R_{2312})+\sqrt{-1}(R_{1312}+R_{2412})],
\end{eqnarray*}
and
\begin{eqnarray*}
&&\mathbf{R}_{\bar{1}22\bar{2}}=R(\overline{u_1}, u_2, u_2, \overline{u_2})\\
&&~~~~~~~~~=\frac{1}{4}R(e_1+\sqrt{-1}e_2,e_3-\sqrt{-1}e_4,e_3-\sqrt{-1}e_4,e_3+\sqrt{-1}e_4)\\
&&~~~~~~~~~=\frac{1}{2}[(R_{1434}-R_{2334})+\sqrt{-1}(R_{1334}+R_{2434})].
\end{eqnarray*}
It follows that $\mathbf{R}_{\bar{1}21\bar{1}}+\mathbf{R}_{\bar{1}22\bar{2}}=0$ on $Q$ is equivalent to
$$R_{1412}-R_{2312}+R_{1434}-R_{2334}=0,$$
$$R_{1312}+R_{2412}+R_{1334}+R_{2434}=0,$$
 i.e. the Ricci tensor of the Levi-Civita connection on  $(M, J, h)$ satisfies
\begin{equation}
Ric(e_2, e_4)+Ric(e_1,e_3)=0,~~Ric(e_2, e_3)-Ric(e_1, e_4)=0,
\end{equation}
for any $J$-adapted orthonormal frame $(e_1, e_2=Je_1, e_3, e_4=Je_3)$.

For a K\"{a}hler surface  $(M, J, h)$, the Ricci tensor is $J$-invariant, so the above equations imply
\begin{equation}
Ric(e_1,e_3)=Ric(e_1, e_4)=0,~~Ric(e_2,e_3)=Ric(e_2, e_4)=0,
\end{equation}
for any $J$-adapted orthonormal frame $(e_1, e_2=Je_1, e_3, e_4=Je_3)$.

Given a $J$-adapted orthonormal frame $(e_1, e_2=Je_1, e_3, e_4=Je_3)$,   we can define a new $J$-adapted orthonormal frame $\frac{1}{\sqrt{2}}(e_1+e_3, e_2+e_4,e_1-e_3, e_2-e_4)$. From (3.27), we have
$$Ric(e_1+e_3,e_1-e_3)=Ric(e_1,e_1)-Ric(e_3,e_3)=0,$$
$$Ric(e_1+e_3,e_2+e_4)=Ric(e_1,e_2)+Ric(e_3,e_4)=0,$$
$$Ric(e_1+e_3,e_2-e_4)=Ric(e_1,e_2)-Ric(e_3,e_4)=0.$$
Together with the fact that the Ricci tensor is $J$-invariant, it follows
\begin{equation}
Ric(e_1,e_2)=Ric(e_3,e_4)=0,
\end{equation}
\begin{equation}
Ric(e_1,e_1)=Ric(e_3,e_3)=Ric(e_2,e_2)=Ric(e_4,e_4).
\end{equation}
Therefore, from equations (3.27), (3.28) and (3.29), it is easy to see that  $(M, J, h)$ is Einstein.

\end{proof}

At the end of this section, for some special Hermitian surface $(M, J, h)$, we show the $\sqrt{-1}\partial\bar{\partial}$-formulas of $\mathbb{K}_1^L(\lambda)$, $\mathbb{K}_3^L(\lambda)$, $\mathbb{K}_4^L(\lambda)$.

Let $(M, J, h)$ be a self-dual Hermitian surface. If the scalar curvature $s$ of $(M, J, h)$ is constant, then from (3.12), (3.7), and the structure equations on $Q$, we have
\begin{eqnarray*}
&&\sqrt{-1}\partial\bar{\partial}\mathbb{K}_1^L(\lambda)=(\sqrt{-1})^2(2-\lambda^2\frac{s}{6})(-\frac{s}{12}\varphi^1\wedge\overline{\varphi^1}\wedge \overline{\varphi^2}\wedge \varphi^2\\
&&~~~~~~~~~~~~~~~~~~~~~~~~~~~~+\overline{\varphi^2}\wedge \varphi^2\wedge\varphi^3\wedge \overline{\varphi^3}
+\varphi^3\wedge \overline{\varphi^3}\wedge \varphi^1\wedge\overline{\varphi^1})\\
&&~~~~~~~~~~~~~~~~~~~~~~~~~~~~+\lambda^2[\sqrt{-1}(\Omega_2^1-\Omega_4^3)\wedge\varphi^3\wedge \overline{\varphi^3}-\tau^3\wedge\overline{\tau^3}],
\end{eqnarray*}
where $\tau^3=\frac{1}{2}[(\Omega_3^1+\Omega_4^2)+\sqrt{-1}(\Omega_3^2-\Omega_4^1)]$. We observe that if the scalar curvature  $s$ is negative, then $\sqrt{-1}\partial\bar{\partial}\mathbb{K}_1^L(\lambda)$ is a positive $(2,2)$-form on the twistor space $(Z, \mathbb{J}_1^L)$, for any positive and sufficiently small $\lambda$.

If the Ricci tensor of a Hermitian surface $(M, J, h)$ is $J$-invariant, from (3.19), (3.20), and the structure equations on $Q$, then
\begin{eqnarray*}
&&\sqrt{-1}\partial\bar{\partial}\mathbb{K}_3^L(\lambda)=
\sqrt{-1}\partial\bar{\partial}\mathbb{K}_4^L(\lambda)\\
&&~~~~~~~~~~~~~~~~~~~~~~~~=\frac{1}{4}(s-s^\ast)\varphi^1\wedge\overline{\varphi^1}\wedge\varphi^2\wedge \overline{\varphi^2}\\
&&~~~~~~~~~~~~~~~~~~~~~~~~~~~+2(\sqrt{-1})^2\mu\wedge\bar{\mu}\wedge(\varphi^1\wedge\overline{\varphi^1}+\varphi^2\wedge \overline{\varphi^2})\\
&&~~~~~~~~~~~~~~~~~~~~~~~~~~~+\lambda^2[\sqrt{-1}(\Omega_2^1-\Omega_4^3)\wedge\varphi^3\wedge \overline{\varphi^3}-\tau^3\wedge\overline{\tau^3}],
\end{eqnarray*}
where $s^\ast$ is the $\ast$-scalar curvature of $(M, J, h)$ \cite{Va1}, $\mu=-\frac{1}{2}[(\omega_3^1-\omega_4^2)+\sqrt{-1}(\omega_3^2+\omega_4^1)]$.
In particular, for a K\"{a}hler surface $(M, J, h)$, we have $\mu=0$ and $s=s^\ast$. It follows that
\begin{eqnarray*}
&&\sqrt{-1}\partial\bar{\partial}\mathbb{K}_3^L(\lambda)=
\sqrt{-1}\partial\bar{\partial}\mathbb{K}_4^L(\lambda)\\
&&~~~~~~~~~~~~~~~~~~~~~~~~=\lambda^2[\sqrt{-1}(\Omega_2^1-\Omega_4^3)\wedge\varphi^3\wedge \overline{\varphi^3}-\tau^3\wedge\overline{\tau^3}].
\end{eqnarray*}

\section{Twistor geometry II: induced by the Chern connection}

In this section, we will study the twistor geometry induced by the Chern connection (also called the second canonical connection) on a Hermitian
surface $(M, J, h)$. Firstly, we consider the integrability and conformal property of almost complex structures
$\mathbb{J}_i^{Ch}$, for $i=1, 2, 3, 4$, on $Z=\mathbb{P}(T^{1,0}M)$. Some metric properties of the natural almost Hermitian metrics $(\mathbb{K}_i^{Ch}(\lambda), \mathbb{J}_i^{Ch})$ on $Z=\mathbb{P}(T^{1,0}M)$ are also obtained.

As in section 3, $\varphi=(\varphi^1, \varphi^2)^t$ is the canonical form on the principal bundle $Q=U(M)$, and let $\psi=(\psi_b^a)$ be the Chern connection form. We denote by $\mathbf{T}=(\mathbf{T}^1, \mathbf{T}^2)^t$ and $\Psi=(\Psi_b^a)$ the corresponding torsion form and curvature form on $Q$, respectively. Then the following structure equations hold:
\begin{equation}
d\varphi^a=-\psi_b^a\wedge\varphi^b+\mathbf{T}^a,
\end{equation}
\begin{equation}
d\psi_b^a=-\psi_c^a\wedge\psi_b^c+\Psi_b^a,
\end{equation}
where $\mathbf{T}^1=\mathbf{T}_{12}^1\varphi^1\wedge\varphi^2$,  $\mathbf{T}^2=\mathbf{T}_{12}^2\varphi^1\wedge\varphi^2$.
$\Psi_b^a=\mathbf{K}_{\bar{a}bc\bar{d}}\varphi^c\wedge\overline{\varphi^d}$,
and $\mathbf{K}_{\bar{a}bc\bar{d}}=K(\overline{u_a}, u_b, u_c, \overline{u_d})$,  $(u_1, u_2)\in Q$.

Now, we can define four natural almost complex structures, denoted by $\mathbb{J}_i^{Ch}$, on $Z=\mathbb{P}(T^{1,0}M)$ as follows. Set $\varphi^3=\psi_2^1$. Here for convenience, we still use the notation $\varphi^3$ as in previous section, but with a different meaning.

~~~~~~~~~~~~~~~~~~~~~~~~~~~~~~~$\mathbb{J}_1^{Ch}$:~~a basis of $(1, 0)$-forms is $\{\varphi^1, \overline{\varphi^2}, \varphi^3\}$;

~~~~~~~~~~~~~~~~~~~~~~~~~~~~~~~$\mathbb{J}_2^{Ch}$:~~a basis of $(1, 0)$-forms is $\{\varphi^1, \overline{\varphi^2}, \overline{\varphi^3}\}$;

~~~~~~~~~~~~~~~~~~~~~~~~~~~~~~~$\mathbb{J}_3^{Ch}$:~~a basis of $(1, 0)$-forms is $\{\varphi^1, \varphi^2, \varphi^3\}$;

~~~~~~~~~~~~~~~~~~~~~~~~~~~~~~~$\mathbb{J}_4^{Ch}$:~~a basis of $(1, 0)$-forms is $\{\varphi^1, \varphi^2, \overline{\varphi^3}\}$.

\noindent A natural family of $\mathbb{J}_i^{Ch}$-compatible Riemannian metrics, denoted by $h_\lambda^{Ch}$, on $Z=\mathbb{P}(T^{1,0}M)$ is
\begin{equation}
h_\lambda^{Ch}=u^\ast(\varphi^1\cdot\overline{\varphi^1}+\varphi^2\cdot\overline{\varphi^2}+\lambda^2\varphi^3\cdot\overline{\varphi^3}),
\end{equation}
where parameter $\lambda>0$, $u$ is a local section of the fibration $\pi_1: Q\rightarrow Z$. The associated fundamental 2-forms are denoted by $\mathbb{K}_i^{Ch}(\lambda)$.

\begin{prop}
Let $(M, J, h)$ be a Hermitian surface. For the Chern connection  on $Q=U(M)$, the almost complex structures $\mathbb{J}_1^{Ch}$, $\mathbb{J}_2^{Ch}$, $\mathbb{J}_3^{Ch}$ and $\mathbb{J}_4^{Ch}$ on the twistor space $Z=\mathbb{P}(T^{1,0}M)$ have the following properties:
\begin{enumerate}[\upshape (i)]
\item $\mathbb{J}_1^{Ch}=\mathbb{J}_1^L$. Thus $\mathbb{J}_1^{Ch}$ is integrable if and only if $(M, J, h)$ is self-dual;
\item $\mathbb{J}_2^{Ch}$ is not integrable;
\item \cite{Va2} $\mathbb{J}_3^{Ch}$ and $\mathbb{J}_4^{Ch}$ are integrable.
\end{enumerate}
\end{prop}

\begin{proof}
On the principal bundle  $Q$, the Lichnerowicz connection form $\phi$ and the Chern connection form $\psi$ have the following relations:
$$\phi-\psi:=\left(\begin{array}{cc}\gamma_1^1&\gamma_2^1\\
\gamma_1^2&\gamma_2^2\end{array}\right),$$
where
$$\gamma_1^1=\frac{1}{2}(\mathbf{T}_{12}^1\varphi^2-\overline{\mathbf{T}_{12}^1}\overline{\varphi^2}),
~~\gamma_2^1=\frac{1}{2}(\mathbf{T}_{21}^1\varphi^1-\overline{\mathbf{T}_{12}^2}\overline{\varphi^2}),$$
$$\gamma_1^2=\frac{1}{2}(\mathbf{T}_{12}^2\varphi^2-\overline{\mathbf{T}_{21}^1}\overline{\varphi^1}),
~~\gamma_2^2=\frac{1}{2}(\mathbf{T}_{21}^2\varphi^1-\overline{\mathbf{T}_{21}^2}\overline{\varphi^1}).$$

(i) Since $\phi_2^1-\psi_2^1=\frac{1}{2}(\mathbf{T}_{21}^1\varphi^1-\overline{\mathbf{T}_{12}^2}\overline{\varphi^2})$, from the constructions of
$\mathbb{J}_1^{Ch}$ and $\mathbb{J}_1^L$, it follows that $\mathbb{J}_1^{Ch}=\mathbb{J}_1^L$. Thus, by using the Proposition 3.1, we have  $\mathbb{J}_1^{Ch}$ is integrable if and only if $(M, J, h)$ is self-dual. 

The integrability of $\mathbb{J}_1^{Ch}$ can also be proved directly from the structure equations. In fact,
$\mathbb{J}_1^{Ch}$ is integrable if and only if $\mathbf{K}_{\bar{1}22\bar{1}}=0$, i.e. $\mathbf{K}_{\bar{1}2\bar{1}2}=0$ on  $Q$. From the curvature relation (2.18), we obtain that $\mathbf{K}_{\bar{1}2\bar{1}2}=0$ if and only if $\mathbf{R}_{\bar{1}2\bar{1}2}=0$  on $Q$. Hence, the result follows from a claim in the proof of Proposition 3.1 (i).

(ii) From the structure equations of the Chern connection on $Q$, it is obvious that $\mathbb{J}_2^{Ch}$ is not integrable. 

(iii) Indeed, for a Hermitian surface $(M, J, h)$, the Chern connection is the unique Hermitian connection with $\mathbf{T}^{1,1}=0$. Meanwhile, the corresponding curvature forms are of $J$-type $(1,1)$. Then the result is also from the structure equations of the Chern connection on $Q$.

\end{proof}

\begin{rem}
For canonical Hermitian connections $\{D^t\}$ in (2.16), the connection forms on the principal bundle $Q=U(M)$ are
$$\phi-t(\phi-\psi)=\left(\begin{array}{cc}\phi_1^1&\phi_2^1\\
\phi_1^2&\phi_2^2\end{array}\right)-t\left(\begin{array}{cc}\gamma_1^1&\gamma_2^1\\
\gamma_1^2&\gamma_2^2\end{array}\right).$$
From the construction of almost complex structures $\mathbb{J}_i$ on $Z=\mathbb{P}(T^{1,0}M)$ in section 2, we observe that canonical Hermitian connections $\{D^t\}$
induce the same $\mathbb{J}_1$, but in general, different $\mathbb{J}_2$, $\mathbb{J}_3$, $\mathbb{J}_4$, for various $t$. In fact, it is natural to study the twistor geometry associated with this family of canonical Hermitian connections.

\end{rem}

Let's to see how the almost complex structures $\mathbb{J}_i^{Ch}$ on $Z$ depend on the choice of metric $h$ in the conformal class $[h]$. For
a Hermitian surface $(M, J, \tilde{h}=e^{2f}h)$ with a conformally related metric $\tilde{h}$, let $\tilde{\pi}: \tilde{Q}\rightarrow M$  be the $U(2)$-principal
bundle of $\tilde{h}$-unitary frames on $M$. The corresponding canonical form,  Chern connection form and torsion form  on $\tilde{Q}$ are denoted by $\tilde{\varphi}=(\tilde{\varphi}^1, \tilde{\varphi}^2)^t$, $\tilde{\psi}=(\tilde{\psi}_b^a)$, $\tilde{\mathbf{T}}=(\tilde{\mathbf{T}}^1, \tilde{\mathbf{T}}^2)^t$, respectively. Then
$$d\tilde{\varphi}^a=-\tilde{\psi}_b^a\wedge\tilde{\varphi}^b+\tilde{\mathbf{T}}^a.$$
As in section 3, we have a bundle isomorphism
$$\mathfrak{f}: \tilde{Q}\rightarrow Q,~~~~~\mathfrak{f}(x, \tilde{u}):=(x, e^f\tilde{u}).$$
Thus $\mathfrak{f}^\ast\varphi=e^{-f}\tilde{\varphi}$. Set $df=f_b\tilde{\varphi}^b+\overline{f_b\tilde{\varphi}^b}$. Then by the structure equations of $\psi$ and $\tilde{\psi}$, we obtain
\begin{equation}
\mathfrak{f}^\ast\psi_b^a=\tilde{\psi}_b^a+\overline{f_c\tilde{\varphi}^c}\delta_b^a-f_c\tilde{\varphi}^c\delta_b^a.
\end{equation}
In particular, $\mathfrak{f}^\ast\psi_2^1=\tilde{\psi}_2^1$. Therefore, from the constructions of almost complex structures $\mathbb{J}_i^{Ch}$ on $Z$,
and the following commutative diagram,
$$\xymatrix{
  \tilde{Q} \ar[rr]^{\mathfrak{f}} \ar[dr]_{\tilde{\pi}_1}
                &  &    Q\ar[dl]^{\pi_1}    \\
                & Z                },
$$
we obtain
\begin{prop}
$\mathbb{J}_i^{Ch}$, for $i=1, 2, 3, 4$, are all conformally invariant.
\end{prop}

Next, we study some metric properties of $(Z=\mathbb{P}(T^{1,0}M), \mathbb{J}_i^{Ch}, \mathbb{K}_i^{Ch}(\lambda))$. The following Theorem is included in Theorem 3.3 and Theorem 4.8 of \cite{Va2}. For the completeness of the present paper, we also give the proof.

\begin{thm}\cite{Va2} Let $(M, J, h)$ be a Hermitian surface. For the Chern connection  on $Q=U(M)$, the almost Hermitian structures $(\mathbb{J}_i^{Ch}, \mathbb{K}_i^{Ch}(\lambda))$ on the twistor space $Z=\mathbb{P}(T^{1,0}M)$ have the following properties:
\begin{enumerate}[\upshape (i)]
\item $d\mathbb{K}_1^{Ch}(\lambda)=0$ if and only if $(M, J, h)$ is a K\"{a}hler surface with constant holomorphic sectional curvature $\frac{4}{\lambda^2}$;
\item $d\mathbb{K}_2^{Ch}(\lambda)=0$ if and only if $(M, J, h)$ is a K\"{a}hler surface with constant holomorphic sectional curvature $-\frac{4}{\lambda^2}$;
\item $d\mathbb{K}_3^{Ch}(\lambda)=0$ (or equivalently, $d\mathbb{K}_4^{Ch}(\lambda)=0$) if and only if $(M, J, h)$ is a flat K\"{a}hler surface.
\end{enumerate}
\end{thm}

\begin{proof}
(i) For the exterior differentiation of $\mathbb{K}_1^{Ch}(\lambda)$, we have
\begin{align}\label{E:1}
d\mathbb{K}_1^{Ch}(\lambda)
                      &=\sqrt{-1}[2\overline{\varphi^1}\wedge\varphi^2\wedge\varphi^3-2\varphi^1\wedge\overline{\varphi^2}\wedge\overline{\varphi^3}\notag\\
                      &~~~~~+\mathbf{T}^1\wedge\overline{\varphi^1}-\overline{\mathbf{T}^1}\wedge\varphi^1+\overline{\mathbf{T}^2}\wedge\varphi^2\notag\\
                      &~~~~~-\mathbf{T}^2\wedge\overline{\varphi^2}
                      +\lambda^2(\Psi_2^1\wedge\overline{\varphi^3}-\overline{\Psi_2^1}\wedge\varphi^3)].
\end{align}
Thus $d\mathbb{K}_1^{Ch}(\lambda)=0$ if and only if
$$\mathbf{T}^1=\mathbf{T}^2=0,~~2\varphi^1\wedge\overline{\varphi^2}=\lambda^2\Psi_2^1$$
on $Q$.

In fact, $\mathbf{T}^1=\mathbf{T}^2=0$ is equivalent to  $(M, J, h)$ is a K\"{a}hler surface. For K\"{a}hler surface, $2\varphi^1\wedge\overline{\varphi^2}=\lambda^2\Psi_2^1$ if and only if
$$\mathbf{R}_{\bar{1}21\bar{2}}=\frac{2}{\lambda^2},~~\mathbf{R}_{\bar{1}21\bar{1}}=\mathbf{R}_{\bar{1}22\bar{2}}=\mathbf{R}_{\bar{1}22\bar{1}}=0.$$
Because of the random  of the unitary frames, the above equations imply
$$\mathbf{R}_{\bar{1}11\bar{1}}=\mathbf{R}_{\bar{2}22\bar{2}}=\frac{4}{\lambda^2}, ~~\mathbf{R}_{\bar{1}12\bar{2}}=\frac{2}{\lambda^2}.$$
Therefore, combining the above discussions, for K\"{a}hler surface, $2\varphi^1\wedge\overline{\varphi^2}=\lambda^2\Psi_2^1$ if and only if
$\mathbf{R}_{\bar{a}bc\bar{d}}=\frac{2}{\lambda^2}(\delta_{ac}\delta_{bd}+\delta_{ab}\delta_{cd})$
on $Q$, i.e. $(M, J, h)$ is a K\"{a}hler surface with constant holomorphic sectional curvature $\frac{4}{\lambda^2}$.

(ii) For the exterior differentiation of $\mathbb{K}_2^{Ch}(\lambda)$, we have
\begin{align}\label{E:1}
d\mathbb{K}_2^{Ch}(\lambda)
                      &=\sqrt{-1}[2\overline{\varphi^1}\wedge\varphi^2\wedge\varphi^3-2\varphi^1\wedge\overline{\varphi^2}\wedge\overline{\varphi^3}\notag\\
                      &~~~~~+\mathbf{T}^1\wedge\overline{\varphi^1}-\overline{\mathbf{T}^1}\wedge\varphi^1+\overline{\mathbf{T}^2}\wedge\varphi^2\notag\\
                      &~~~~~-\mathbf{T}^2\wedge\overline{\varphi^2}
                      +\lambda^2(\overline{\Psi_2^1}\wedge\varphi^3-\Psi_2^1\wedge\overline{\varphi^3})].
\end{align}
Thus $d\mathbb{K}_2^{Ch}(\lambda)=0$ if and only if
$$\mathbf{T}^1=\mathbf{T}^2=0,~~2\varphi^1\wedge\overline{\varphi^2}+\lambda^2\Psi_2^1=0$$
on $Q$.

Therefore, the subsequent proof is the same as in (i).

(iii) For the exterior differentiation of $\mathbb{K}_3^{Ch}(\lambda)$ and $\mathbb{K}_3^{Ch}(\lambda)$, we have
\begin{align}\label{E:1}
d\mathbb{K}_3^{Ch}(\lambda)
                      &=\sqrt{-1}[\mathbf{T}^1\wedge\overline{\varphi^1}-\overline{\mathbf{T}^1}\wedge\varphi^1+\mathbf{T}^2\wedge\overline{\varphi^2}\notag\\
                      &~~~~~-\overline{\mathbf{T}^2}\wedge\varphi^2
                      +\lambda^2(\Psi_2^1\wedge\overline{\varphi^3}-\overline{\Psi_2^1}\wedge\varphi^3)],
\end{align}
\begin{align}\label{E:1}
d\mathbb{K}_4^{Ch}(\lambda)
                      &=\sqrt{-1}[\mathbf{T}^1\wedge\overline{\varphi^1}-\overline{\mathbf{T}^1}\wedge\varphi^1+\mathbf{T}^2\wedge\overline{\varphi^2}\notag\\
                      &~~~~~-\overline{\mathbf{T}^2}\wedge\varphi^2
                      +\lambda^2(\overline{\Psi_2^1}\wedge\varphi^3-\Psi_2^1\wedge\overline{\varphi^3})].
\end{align}
Thus $d\mathbb{K}_3^{Ch}(\lambda)=0$ (equivalently, $d\mathbb{K}_4^{Ch}(\lambda)=0$) if and only if
$$\mathbf{T}^1=\mathbf{T}^2=0,~~\Psi_2^1=0$$
on $Q$.

Obviously, $\mathbf{T}^1=\mathbf{T}^2=0$ is equivalent to  $(M, J, h)$ is a K\"{a}hler surface.
For K\"{a}hler surface, the curvature forms on $Q$ satisfy $\Omega_3^1=\Omega_4^2$, $\Omega_4^1=-\Omega_3^2$, and then
$\Psi_2^1=\Omega_3^1+\sqrt{-1}\Omega_3^2$. Thus $\Psi_2^1=0$ if and only if $\Omega_3^1=\Omega_3^2=\Omega_4^1=\Omega_4^2=0$. We have a claim in the proof of Theorem 3.4 (iii) that $\Omega_3^1=\Omega_4^2=\Omega_4^1=\Omega_3^2=0$ on principal bundle $Q$ if and only if $(M, J, h)$ is flat. So we obtain the result.

\end{proof}

\begin{thm} Let $(M, J, h)$ be a Hermitian surface. For the Chern connection  on $Q=U(M)$, the almost Hermitian structures $(\mathbb{J}_i^{Ch}, \mathbb{K}_i^{Ch}(\lambda))$ on the twistor space $Z=\mathbb{P}(T^{1,0}M)$ have the following properties:
\begin{enumerate}[\upshape (i)]
\item $d\mathbb{K}_1^{Ch}(\lambda)\wedge\mathbb{K}_1^{Ch}(\lambda)=0$ (or equivalently, $d\mathbb{K}_2^{Ch}(\lambda)\wedge\mathbb{K}_2^{Ch}(\lambda)=0$) if and only if $(M, J, h)$ is a self-dual K\"{a}hler surface;
\item $d\mathbb{K}_3^{Ch}(\lambda)\wedge\mathbb{K}_3^{Ch}(\lambda)=0$ (or equivalently, $d\mathbb{K}_4^{Ch}(\lambda)\wedge\mathbb{K}_4^{Ch}(\lambda)=0$) if and only if $(M, J, h)$ is a K\"{a}hler Einstein surface.
\end{enumerate}
\end{thm}

\begin{proof}
(i) From the exterior differential formula (4.5), we have
\begin{align}\label{E:1}
-\mathbb{K}_1^{Ch}(\lambda)\wedge d\mathbb{K}_1^{Ch}(\lambda)
                      &=\lambda^2[(\varphi^1\wedge\overline{\varphi^1}+\overline{\varphi^2}\wedge\varphi^2)\wedge \Psi_2^1\wedge\overline{\varphi^3}\notag\\
                      &~~~~-(\varphi^1\wedge\overline{\varphi^1}+\overline{\varphi^2}\wedge\varphi^2)\wedge \overline{\Psi_2^1}\wedge\varphi^3\notag\\
                      &~~~~+(\mathbf{T}^1\wedge\overline{\varphi^1}-\overline{\mathbf{T}^1}\wedge\varphi^1)\wedge \varphi^3\wedge\overline{\varphi^3}\notag\\
                      &~~~~+(\overline{\mathbf{T}^2}\wedge\varphi^2
                     -\mathbf{T}^2\wedge\overline{\varphi^2})\wedge \varphi^3\wedge\overline{\varphi^3}].
\end{align}

From the exterior differential formula (4.6), we have
\begin{align}\label{E:1}
-\mathbb{K}_2^{Ch}(\lambda)\wedge d\mathbb{K}_2^{Ch}(\lambda)
                      &=\lambda^2[(\varphi^1\wedge\overline{\varphi^1}+\overline{\varphi^2}\wedge\varphi^2)\wedge \overline{\Psi_2^1}\wedge\varphi^3\notag\\
                      &~~~~-(\varphi^1\wedge\overline{\varphi^1}+\overline{\varphi^2}\wedge\varphi^2)\wedge\Psi_2^1\wedge\overline{\varphi^3}\notag\\
                      &~~~~+(\mathbf{T}^1\wedge\overline{\varphi^1}-\overline{\mathbf{T}^1}\wedge\varphi^1)\wedge \overline{\varphi^3}\wedge\varphi^3\notag\\
                      &~~~~+(\overline{\mathbf{T}^2}\wedge\varphi^2
                     -\mathbf{T}^2\wedge\overline{\varphi^2})\wedge \overline{\varphi^3}\wedge\varphi^3].
\end{align}

It follows that $d\mathbb{K}_1^{Ch}(\lambda)\wedge\mathbb{K}_1^{Ch}(\lambda)=0$ (or equivalently,
$d\mathbb{K}_2^{Ch}(\lambda)\wedge\mathbb{K}_2^{Ch}(\lambda)=0$) if and only if
$\mathbf{T}^1=\mathbf{T}^2=0$ and $\mathbf{K}_{\bar{1}22\bar{2}}=\mathbf{K}_{\bar{1}21\bar{1}}$ on $Q$.
As we have observed that $\mathbf{T}^1=\mathbf{T}^2=0$ if and only if $(M, J, h)$ is a K\"{a}hler surface. For K\"{a}hler surface,
the Chern connection coincides with the Levi-Civita connection, so $\mathbf{K}_{\bar{1}22\bar{2}}=\mathbf{K}_{\bar{1}21\bar{1}}$  if and only if
$\mathbf{R}_{\bar{1}22\bar{2}}=\mathbf{R}_{\bar{1}21\bar{1}}$  on $Q$. Now, the result follows from the proof of Theorem 3.5 (i).

(ii) From the exterior differential formula (4.7), we have
\begin{align}\label{E:1}
-\mathbb{K}_3^{Ch}(\lambda)\wedge d\mathbb{K}_3^{Ch}(\lambda)
                      &=\lambda^2[(\varphi^1\wedge\overline{\varphi^1}+\varphi^2\wedge\overline{\varphi^2})\wedge \Psi_2^1\wedge\overline{\varphi^3}\notag\\
                      &~~~~-(\varphi^1\wedge\overline{\varphi^1}+\varphi^2\wedge\overline{\varphi^2})\wedge\overline{\Psi_2^1}\wedge\varphi^3\notag\\
                      &~~~~+(\mathbf{T}^1\wedge\overline{\varphi^1}-\overline{\mathbf{T}^1}\wedge\varphi^1)\wedge \varphi^3\wedge\overline{\varphi^3}\notag\\
                      &~~~~+(\mathbf{T}^2\wedge\overline{\varphi^2}
                     -\overline{\mathbf{T}^2}\wedge\varphi^2)\wedge \varphi^3\wedge\overline{\varphi^3}].
\end{align}

From the exterior differential formula (4.8), we have
\begin{align}\label{E:1}
-\mathbb{K}_4^{Ch}(\lambda)\wedge d\mathbb{K}_4^{Ch}(\lambda)
                      &=\lambda^2[(\varphi^1\wedge\overline{\varphi^1}+\varphi^2\wedge\overline{\varphi^2})\wedge \overline{\Psi_2^1}\wedge\varphi^3\notag\\
                      &~~~~-(\varphi^1\wedge\overline{\varphi^1}+\varphi^2\wedge\overline{\varphi^2})\wedge\Psi_2^1\wedge\overline{\varphi^3}\notag\\
                      &~~~~+(\mathbf{T}^1\wedge\overline{\varphi^1}-\overline{\mathbf{T}^1}\wedge\varphi^1)\wedge \overline{\varphi^3}\wedge\varphi^3\notag\\
                      &~~~~+(\mathbf{T}^2\wedge\overline{\varphi^2}
                      -\overline{\mathbf{T}^2}\wedge\varphi^2)\wedge \overline{\varphi^3}\wedge\varphi^3].
\end{align}

It follows that $d\mathbb{K}_3^{Ch}(\lambda)\wedge\mathbb{K}_3^{Ch}(\lambda)=0$ (or equivalently,
$d\mathbb{K}_4^{Ch}(\lambda)\wedge\mathbb{K}_4^{Ch}(\lambda)=0$) if and only if
$\mathbf{T}^1=\mathbf{T}^2=0$ and $\mathbf{K}_{\bar{1}22\bar{2}}+\mathbf{K}_{\bar{1}21\bar{1}}=0$ on $Q$.
As we have observed that $\mathbf{T}^1=\mathbf{T}^2=0$ if and only if $(M, J, h)$ is a K\"{a}hler surface. For K\"{a}hler surface,
the Chern connection coincides with the Levi-Civita connection, so $\mathbf{K}_{\bar{1}22\bar{2}}+\mathbf{K}_{\bar{1}21\bar{1}}=0$  if and only if
$\mathbf{R}_{\bar{1}22\bar{2}}+\mathbf{R}_{\bar{1}21\bar{1}}=0$  on $Q$. Therefore, the result follows from the proof of Theorem 3.5 (ii).

\end{proof}

\begin{rem}
The metric condition on $(Z=\mathbb{P}(T^{1,0}M)$, $\mathbb{J}_i^L, \mathbb{K}_i^L(\lambda))$ in Theorem 3.5 and on $(Z=\mathbb{P}(T^{1,0}M)$, $\mathbb{J}_i^{Ch}, \mathbb{K}_i^{Ch}(\lambda))$ in Theorem 4.5 is called the balanced metric condition. M. L. Michelsohn \cite{Mic} introduced the balanced metric on complex manifold, and also claimed that the natural metric on the twistor space $(Z,\mathbb{J}_+)$ of a self-dual 4-manifold is a balanced metric. This metric is systematically studied in \cite{JR,FZ} by using the moving frames method.
\end{rem}

\begin{rem}
For a K\"{a}hler surface $(M, J, h)$, the canonical Hermitian metrics $\{D^t\}$ degenerate to a single point, the Levi-Civita connection. The induced natural geometry structures  on $Z=\mathbb{P}(T^{1,0}M)$ are denoted by $(\mathbb{J}_i^{LC},\mathbb{K}_i^{LC}(\lambda))$, $i=1,2,3,4$. In this case, $\mathbb{J}_1^{LC}=\mathbb{J}_+$ and $\mathbb{J}_2^{LC}=\mathbb{J}_-$. I. Vaisman\cite{Va2} showed that for a K\"{a}hler surface $(M, J, h)$ with constant holomorphic sectional curvature, the torsion 1-forms of the corresponding
metrics $\mathbb{K}_i^{LC}(\lambda)$ on the twistor space are zero, that is all $\mathbb{K}_i^{LC}(\lambda)$ satisfy the balanced metric condition.
\end{rem}

At the end of this section, we show the following $\sqrt{-1}\partial\bar{\partial}$-formulas of  $\mathbb{K}_3^{Ch}(\lambda)$ and $\mathbb{K}_4^{Ch}(\lambda)$:
\begin{eqnarray*}
&&\sqrt{-1}\partial\bar{\partial}\mathbb{K}_3^{Ch}(\lambda)=
\sqrt{-1}\partial\bar{\partial}\mathbb{K}_4^{Ch}(\lambda)\\
&&~~~~~~~~~~~~~~~~~~~~~~~~~=-\lambda^2[\Psi_2^1\wedge\overline{\Psi_2^1}+(\Psi_1^1-\Psi_2^2)\wedge\varphi^3\wedge\overline{\varphi^3}]\\
&&~~~~~~~~~~~~~~~~~~~~~~~~~~~~~+\Psi_1^1\wedge\varphi^1\wedge\overline{\varphi^1}+\Psi_2^2\wedge\varphi^2\wedge\overline{\varphi^2}
+\mathbf{T}^1\wedge\overline{\mathbf{T}^1}\\
&&~~~~~~~~~~~~~~~~~~~~~~~~~~~~~-\overline{\Psi_2^1}\wedge\varphi^1\wedge\overline{\varphi^2}-\Psi_2^1\wedge\overline{\varphi^1}\wedge\varphi^2
+\mathbf{T}^2\wedge\overline{\mathbf{T}^2}.
\end{eqnarray*}

\section{Conclusions and discussions}
By using the method of moving frames, we give a comprehensive study of the four natural almost Hermitian
structures on the twistor space $Z=\mathbb{P}(T^{1,0}M)$ associated with the Lichnerowicz connection and the Chern connection on a  Hermitian surface $(M, J, h)$, respectively. For the Lichnerowicz connection, the induced natural almost Hermitian
structures on the twistor space $Z=\mathbb{P}(T^{1,0}M)$ are denoted by $(\mathbb{J}_i^L, \mathbb{K}_i^L(\lambda))$, $i=1,2,3,4$. We consider the integrability and conformal property of $\mathbb{J}_i^L$.  In particular, we prove that $\mathbb{J}_3^L$ (or $\mathbb{J}_4^L$) is integrable if and only if the Ricci tensor of the Levi-Civita connection on  $(M, J, h)$ is $J$-invariant. The symplectic metric condition and the balanced metric condition of the natural almost Hermitian metrics $\mathbb{K}_i^L(\lambda)$ on $Z=\mathbb{P}(T^{1,0}M)$ are studied. For some special Hermitian surface $(M, J, h)$, we show the $\sqrt{-1}\partial\bar{\partial}$-formulas of $\mathbb{K}_1^L(\lambda)$, $\mathbb{K}_3^L(\lambda)$, $\mathbb{K}_4^L(\lambda)$.  For the Chern connection, the induced natural almost Hermitian structures on the twistor space $Z=\mathbb{P}(T^{1,0}M)$ are denoted by $(\mathbb{J}_i^{Ch}, \mathbb{K}_i^{Ch}(\lambda))$, $i=1,2,3,4$. We prove that $\mathbb{J}_1^{Ch}=\mathbb{J}_1^L$, and all $\mathbb{J}_i^{Ch}$ are conformally invariant. We also consider the symplectic metric condition and the balanced metric condition of the natural almost Hermitian metrics $\mathbb{K}_i^{Ch}(\lambda)$ on $Z=\mathbb{P}(T^{1,0}M)$.

In fact, as we show in section 2, associated with any unitary connection on a Hermitian surface $(M, J, h)$, we can define 3-parameter family of almost Hermitian structures, denoted by $\mathbb{K}_i(\lambda_1,\lambda_2,\lambda_3)$,  on the twistor space  $Z=\mathbb{P}(T^{1,0}M)$. Thus, for the Lichnerowicz connection and the Chern connection, we obtain many almost Hermitian structures, denoted by  $(\mathbb{J}_i^L, \mathbb{K}_i^L(\lambda_1,\lambda_2,\lambda_3))$ and $(\mathbb{J}_i^{Ch}, \mathbb{K}_i^{Ch}(\lambda_1,\lambda_2,\lambda_3))$ on $Z=\mathbb{P}(T^{1,0}M)$. These may lead to more results. For example, Proposition 3.7 in \cite{Va2} includes in these family of almost Hermitian structures. The similar constructions are in details studied in the appendix for the twistor space of complex projective
plane $\mathbb{C}P^2$ with the Fubini-Study metric $h_{FS}$. We will present the curvature properties of the constructed almost Hermitian structures on the twistor space $Z=\mathbb{P}(T^{1,0}M)$, and the generalization of these structures to the twistor spaces of higher dimensional manifold $M$ (in particular, $M$ is a six dimensional (almost) complex manifold) in our forthcoming papers.

The following two points are related discussions.

1. We want to point out that $Z=\mathbb{P}(T^{1,0}M)$ has a natural Hermitian structure induced from the Hermitian structure on the holomorphic
tangent bundle $T^{1,0}M$ of Hermitian surface $(M, J, h)$. We review the general constructions of natural geometry structures on projective bundles\cite{Voi}.
Let $(M, J, h)$ be a Hermitian manifold, the corresponding fundamental 2-form is denoted by $F$. Let $(E, h_E)$ be a holomorphic vector bundle with a Hermitian structure $h_E$ over $(M, J, h)$. $\mathbb{P}(E)$ is the projective bundle of $E$, the natural projection is denoted by $p: \mathbb{P}(E)\rightarrow M$. Namely, for any point $x\in M$, the fibre $\mathbb{P}(E)_x$ is the projective space $\mathbb{P}(E_x)$ of the fibre $E_x$. Then  $\mathbb{P}(E)$ has a natural complex structure,  meanwhile, $h$ and $h_E$ induce a positive $(1,1)$-form, denoted by $F_{\mathbb{P}(E)}$, on $\mathbb{P}(E)$,
\begin{equation}
F_{\mathbb{P}(E)}=\lambda p^\ast F+\sqrt{-1}\partial\bar\partial\log h_E(v,v),
\end{equation}
where $(x,[v])\in \mathbb{P}(E)$, $\lambda$ is a sufficiently large number.

It is easy to see that if $F$ is a K\"{a}hler metric on $M$, then the above $F_{\mathbb{P}(E)}$ is a K\"{a}hler metric on $\mathbb{P}(E)$. Of course, $E=T^{1,0}M$ is a particular case, and $\mathbb{P}(T^{1,0}M)$ is regarded as the twistor space $Z$ in the present paper. From the results in sections 3 and 4, our almost Hermitian structures on $\mathbb{P}(T^{1,0}M)$ are different from the above construction (5.1).

2. It is worth to mention the results of compatible almost complex structures on the twistor space $Z$ obtained by G. Deschamps\cite{Des}. As introduced in section 1, for an oriented Riemannian 4-manifold $(M, h)$, the Levi-Civita connection on $M$ induces a splitting of the tangent bundle $TZ$ into the direct sum of the horizontal and vertical distributions, denoted by $TZ=\mathcal{H}\oplus \mathcal{V}$. With respect to this decomposition, G. Deschamps\cite{Des} constructed almost complex structures $\mathbb{J}_\Phi$ on $Z$, by using the fibre preserving morphisms $\Phi$:
$$\xymatrix{
  Z \ar[rr]^{\Phi} \ar[dr]_{\pi_2}
                &  &    Z \ar[dl]^{\pi_2}    \\
                & M                 }.
$$
In particular, $\mathbb{J}_{Id}=\mathbb{J}_+$, $\mathbb{J}_{\sigma}=-\mathbb{J}_-$, where $Id$ is the identity morphism of the twistor space $Z$,  $\sigma$ is the morphism of $Z$ whose restriction to the fibre of the twistor fibration is the antipodal map. For a Hermitian surface $(M, J, h)$,  G. Deschamps\cite{Des} introduced another natural almost complex structure $\mathbb{J}_\infty$, and proved that $\mathbb{J}_\infty$ is integrable if and only if $(M, J, h)$ is a K\"{a}hler surface, in the compact case.  Recently, D. Ali, J. Davidov and O. Muskarov\cite{ADM} considered the Gray-Hervella classes of the natural almost Hermitian structures on the twistor space $Z$ introduced by G. Deschamps.

Our viewpoint of almost Hermitian structures on the twistor space is from the canonical connections on the principal bundle $Q=U(M)$. On the other hand, G. Deschamps's construction is from  the fibre preserving morphism of the twistor space. It is an interesting question to study further the similarities and differences of these two viewpoints. Chern numbers \cite{KT} may be the candidates for further study.

\section{Appendix}

In order to have a better understanding of this paper, we make some explicit calculations for the case of complex projective
plane $\mathbb{C}P^2$ with the Fubini-Study metric $h_{FS}$. It is well-known that $(\mathbb{C}P^2, h_{FS})$ is K\"{a}hler-Einstein and self-dual.
The following results are familiar to expert\cite{BH,ES,Yan}. Using the notations as in section 2, we have projection
$$\pi_2: Z=\mathbb{P}(T^{1,0}\mathbb{C}P^2)=\frac{SU(3)}{S(U(1)\times U(1)\times U(1))}\rightarrow \mathbb{C}P^2=\frac{SU(3)}{S(U(1)\times U(2))}.$$

By general principles, $Z=\mathbb{P}(T^{1,0}\mathbb{C}P^2)$ admits $2^3=8$ invariant almost complex structures and exactly $3!=6$ of these are integrable.

The Maurer-Cartan form of $SU(3)$ is denoted by $w=g^{-1}dg=(w_m^l)_{l,m=1,2,3}$, $g\in SU(3)$. $w$ is a $su(3)$-valued 1-form. The exterior differentiation of
$w$ leads to the so called structure equations of $SU(3)$,
$$dw=-w\wedge w.$$
By local section of $\pi: SU(3)\rightarrow \mathbb{C}P^2$, the pull-back of $\{w_2^1, w_3^1\}$ form a local unitary coframe on $(\mathbb{C}P^2, h_{FS})$. Next, we make the calculations on $SU(3)$. From the structure equations, we have
\begin{equation}
dw_2^1=-(w_1^1-w_2^2)\wedge w_2^1-\overline{w_3^2}\wedge w_3^1,
\end{equation}
\begin{equation}
dw_3^1=-(w_1^1-w_3^3)\wedge w_3^1+w_3^2\wedge w_2^1,
\end{equation}
\begin{equation}
dw_3^2=-(w_2^2-w_3^3)\wedge w_3^2+\overline {w_2^1}\wedge w_3^1.
\end{equation}

Set
$$\phi=\left(\begin{array}{cc}w_1^1-w_2^2&\overline{w_3^2}\\-w_3^2&w_1^1-w_3^3
\end{array}\right).$$
Using the notations as in section 2, we define eight natural almost complex structures on $Z=\mathbb{P}(T^{1,0}\mathbb{C}P^2)$ as follows.

~~~~~~~~~~~~~~~~~~~~~~~~~~~~~$\mathbb{J}_1$:~~a basis of $(1, 0)$-forms is $\{w_2^1, \overline{w_3^1}, \overline{w_3^2}\}$;

~~~~~~~~~~~~~~~~~~~~~~~~~~~~~$\mathbb{J}_2$:~~a basis of $(1, 0)$-forms is $\{w_2^1, \overline{w_3^1},  w_3^2\}$;

~~~~~~~~~~~~~~~~~~~~~~~~~~~~~$\mathbb{J}_3$:~~a basis of $(1, 0)$-forms is $\{w_2^1,  w_3^1, \overline{w_3^2}\}$;

~~~~~~~~~~~~~~~~~~~~~~~~~~~~~$\mathbb{J}_4$:~~a basis of $(1, 0)$-forms is $\{w_2^1,  w_3^1,  w_3^2\}$,

\noindent and their conjugations $\mathbb{J}_5:=-\mathbb{J}_1$, $\mathbb{J}_6:=-\mathbb{J}_2$,
$\mathbb{J}_7:=-\mathbb{J}_3$, $\mathbb{J}_8:=-\mathbb{J}_4$. We only need to consider $\mathbb{J}_i$, $i=1, 2, 3, 4$, and the
natural fundamental 2-forms are
\begin{eqnarray*}
&&\mathbb{K}_1(\lambda)=\sqrt{-1}(w_2^1\wedge\overline{w_2^1}+\overline{w_3^1}\wedge w_3^1+\lambda^2\overline{w_3^2}\wedge w_3^2),\\
&&\mathbb{K}_2(\lambda)=\sqrt{-1}(w_2^1\wedge\overline{w_2^1}+\overline{w_3^1}\wedge w_3^1+\lambda^2w_3^2\wedge \overline{w_3^2}),\\
&&\mathbb{K}_3(\lambda)=\sqrt{-1}(w_2^1\wedge\overline{w_2^1}+w_3^1\wedge \overline{w_3^1}+\lambda^2\overline{w_3^2}\wedge w_3^2),\\
&&\mathbb{K}_4(\lambda)=\sqrt{-1}(w_2^1\wedge\overline{w_2^1}+w_3^1\wedge \overline{w_3^1}+\lambda^2w_3^2\wedge \overline{w_3^2}).
\end{eqnarray*}

Now, from the equations (6.1), (6.2) and (6.3), obviously, $\mathbb{J}_1$, $\mathbb{J}_3$ and $\mathbb{J}_4$ are integrable, $\mathbb{J}_2$ is not integrable. For
the first exterior differentiation of $\mathbb{K}_i(\lambda)$, we obtain
\begin{eqnarray*}
&&d\mathbb{K}_1(\lambda)=\sqrt{-1}(2-\lambda^2)(\overline{w_2^1}\wedge w_3^1\wedge\overline{ w_3^2}-w_2^1\wedge \overline{w_3^1}\wedge w_3^2),\\
&&d\mathbb{K}_2(\lambda)=\sqrt{-1}(2+\lambda^2)(\overline{w_2^1}\wedge w_3^1\wedge\overline{ w_3^2}-w_2^1\wedge \overline{w_3^1}\wedge w_3^2),\\
&&d\mathbb{K}_3(\lambda)=\sqrt{-1}\lambda^2(w_2^1\wedge \overline{w_3^1}\wedge w_3^2-\overline{w_2^1}\wedge w_3^1\wedge\overline{ w_3^2}),\\
&&d\mathbb{K}_4(\lambda)=\sqrt{-1}\lambda^2(\overline{w_2^1}\wedge w_3^1\wedge\overline{ w_3^2}-w_2^1\wedge \overline{w_3^1}\wedge w_3^2).
\end{eqnarray*}

\noindent Thus $d\mathbb{K}_1(\lambda)=0$ if and only if $\lambda^2=2$; $d\mathbb{K}_2(\lambda)^{(1,2)}=0$ (i.e. $\mathbb{K}_2(\lambda)$ is $(1, 2)$-symplectic, but not symplectic); $\mathbb{K}_3(\lambda)$ and $\mathbb{K}_4(\lambda)$ are not $d$-closed.

From the above exterior differential
formulas of $\mathbb{K}_i(\lambda)$, it also follows that $\mathbb{K}_i(\lambda)$  all satisfy the balanced metric condition, i.e.
$d\mathbb{K}_i(\lambda)\wedge \mathbb{K}_i(\lambda)=0$, $i=1, 2, 3, 4$.

By direct calculations, we have the following $\sqrt{-1}\partial\bar{\partial}$-formulas of $\mathbb{K}_1(\lambda)$, $\mathbb{K}_3(\lambda)$, $\mathbb{K}_4(\lambda)$:
\begin{eqnarray*}
&&\sqrt{-1}\partial\bar{\partial}\mathbb{K}_1(\lambda)=(\sqrt{-1})^2(2-\lambda^2)(-w_2^1\wedge\overline{w_2^1}\wedge \overline{w_3^1}\wedge w_3^1\\
&&~~~~~~~~~~~~~~~~~~~~~~~~~~~~+\overline{w_3^1}\wedge w_3^1\wedge\overline{w_3^2}\wedge w_3^2+\overline{w_3^2}\wedge w_3^2\wedge w_2^1\wedge\overline{w_2^1}),\\
&&\sqrt{-1}\partial\bar{\partial}\mathbb{K}_3(\lambda)=\sqrt{-1}\partial\bar{\partial}\mathbb{K}_4(\lambda)\\
&&~~~~~~~~~~~~~~~~~~~~~~~=(\sqrt{-1})^2\lambda^2
(w_2^1\wedge\overline{w_2^1}\wedge w_3^1\wedge\overline{w_3^1}-w_3^1\wedge\overline{w_3^1}\wedge\overline{w_3^2}\wedge w_3^2\\
&&~~~~~~~~~~~~~~~~~~~~~~~~~~~~+\overline{w_3^2}\wedge w_3^2\wedge w_2^1\wedge\overline{w_2^1}).
\end{eqnarray*}

In fact, the pull-back of $w_2^1\wedge\overline{w_2^1}$, $w_3^1\wedge\overline{w_3^1}$ and  $w_3^2\wedge\overline{w_3^2}$ are globally defined 2-forms
on $Z=\mathbb{P}(T^{1,0}\mathbb{C}P^2)$. Therefore, we can consider  3-parameter family of natural (almost) Hermitian metrics on $Z$ as follows:
\begin{eqnarray*}
&&\mathbb{K}_1(\lambda_1, \lambda_2, \lambda_3)=\sqrt{-1}
(\lambda_1^2 w_2^1\wedge\overline{w_2^1}+\lambda_2^2\overline{w_3^1}\wedge w_3^1+\lambda_3^2\overline{w_3^2}\wedge w_3^2),\\
&&\mathbb{K}_2(\lambda_1, \lambda_2, \lambda_3)=\sqrt{-1}
(\lambda_1^2 w_2^1\wedge\overline{w_2^1}+\lambda_2^2\overline{w_3^1}\wedge w_3^1+\lambda_3^2 w_3^2\wedge\overline{w_3^2}),\\
&&\mathbb{K}_3(\lambda_1, \lambda_2, \lambda_3)=\sqrt{-1}
(\lambda_1^2 w_2^1\wedge\overline{w_2^1}+\lambda_2^2 w_3^1\wedge\overline{w_3^1}+\lambda_3^2 \overline{w_3^2}\wedge w_3^2),\\
&&\mathbb{K}_4(\lambda_1, \lambda_2, \lambda_3)=\sqrt{-1}
((\lambda_1^2 w_2^1\wedge\overline{w_2^1}+\lambda_2^2 w_3^1\wedge\overline{w_3^1}+\lambda_3^2 w_3^2\wedge\overline{w_3^2}).
\end{eqnarray*}
where three parameters $\lambda_1>0$, $\lambda_2>0$, $\lambda_3>0$. As in $\mathbb{K}_i(\lambda)$, for the first exterior differentiation  of
$\mathbb{K}_i(\lambda_1, \lambda_2, \lambda_3)$, we have
\begin{eqnarray*}
&&d\mathbb{K}_1(\lambda_1, \lambda_2, \lambda_3)
=\sqrt{-1}(\lambda_1^2+\lambda_2^2-\lambda_3^2)(\overline{w_2^1}\wedge w_3^1\wedge\overline{ w_3^2}-w_2^1\wedge \overline{w_3^1}\wedge w_3^2),\\
&&d\mathbb{K}_2(\lambda_1, \lambda_2, \lambda_3)
=\sqrt{-1}(\lambda_1^2+\lambda_2^2+\lambda_3^2)(\overline{w_2^1}\wedge w_3^1\wedge\overline{ w_3^2}-w_2^1\wedge \overline{w_3^1}\wedge w_3^2),\\
&&d\mathbb{K}_3(\lambda_1, \lambda_2, \lambda_3)
=\sqrt{-1}(\lambda_1^2-\lambda_2^2-\lambda_3^2)(\overline{w_2^1}\wedge w_3^1\wedge\overline{ w_3^2}-w_2^1\wedge \overline{w_3^1}\wedge w_3^2),\\
&&d\mathbb{K}_4(\lambda_1, \lambda_2, \lambda_3)
=\sqrt{-1}(\lambda_1^2-\lambda_2^2+\lambda_3^2)(\overline{w_2^1}\wedge w_3^1\wedge\overline{ w_3^2}-w_2^1\wedge \overline{w_3^1}\wedge w_3^2).
\end{eqnarray*}

\noindent Thus $d\mathbb{K}_1(\lambda_1, \lambda_2, \lambda_3)=0$ if and only if $\lambda_1^2+\lambda_2^2=\lambda_3^2$; $d\mathbb{K}_2(\lambda_1, \lambda_2, \lambda_3)^{(1,2)}=0$ (i.e. $\mathbb{K}_2(\lambda_1, \lambda_2, \lambda_3)$ is $(1, 2)$-symplectic, but not symplectic); $d\mathbb{K}_3(\lambda_1, \lambda_2, \lambda_3)=0$ if and only if $\lambda_1^2=\lambda_2^2+\lambda_3^2$; $d\mathbb{K}_4(\lambda_1, \lambda_2, \lambda_3)=0$ if and only if $\lambda_2^2=\lambda_1^2+\lambda_3^2$.

From the above exterior differential
formulas of $\mathbb{K}_i(\lambda_1, \lambda_2, \lambda_3)$, it also follows that $\mathbb{K}_i(\lambda_1, \lambda_2, \lambda_3)$  all satisfy the balanced metric condition, i.e.
$d\mathbb{K}_i(\lambda_1, \lambda_2, \lambda_3)\wedge \mathbb{K}_i(\lambda_1, \lambda_2, \lambda_3)$ $=0$, $i=1, 2, 3, 4$. In particular,
for $\mathbb{K}_2(\frac{1}{\sqrt{2}}, \frac{1}{\sqrt{2}}, \frac{1}{\sqrt{2}})$, we have
\begin{eqnarray*}
&&d\mathbb{K}_2(\frac{1}{\sqrt{2}}, \frac{1}{\sqrt{2}}, \frac{1}{\sqrt{2}})=3\mathrm{Re}(\rho),\\
&&d\mathrm{Im}(\rho)=-2\mathbb{K}_2(\frac{1}{\sqrt{2}}, \frac{1}{\sqrt{2}}, \frac{1}{\sqrt{2}})\wedge \mathbb{K}_2(\frac{1}{\sqrt{2}}, \frac{1}{\sqrt{2}}, \frac{1}{\sqrt{2}}),
\end{eqnarray*}
where $\rho=(\sqrt{-1})^3w_2^1\wedge \overline{w_3^1}\wedge w_3^2$. It is equivalent to say that $\mathbb{K}_2(\frac{1}{\sqrt{2}}, \frac{1}{\sqrt{2}}, \frac{1}{\sqrt{2}})$ is a nearly K\"{a}hler structure on $Z=\mathbb{P}(T^{1,0}\mathbb{C}P^2)$.

By direct calculations, we also have the following $\sqrt{-1}\partial\bar{\partial}$-formulas of $\mathbb{K}_1(\lambda_1, \lambda_2, \lambda_3)$, $\mathbb{K}_3(\lambda_1, \lambda_2, \lambda_3)$, $\mathbb{K}_4(\lambda_1, \lambda_2, \lambda_3)$:
\begin{eqnarray*}
&&\sqrt{-1}\partial\bar{\partial}\mathbb{K}_1(\lambda_1, \lambda_2, \lambda_3)
=(\sqrt{-1})^2(\lambda_1^2+\lambda_2^2-\lambda_3^2)(-w_2^1\wedge\overline{w_2^1}\wedge \overline{w_3^1}\wedge w_3^1\\
&&~~~~~~~~~~~~~~~~~~~~~~~~~~~~~~~~~~~~~~~~~+\overline{w_3^1}\wedge w_3^1\wedge\overline{w_3^2}\wedge w_3^2+\overline{w_3^2}\wedge w_3^2\wedge w_2^1\wedge\overline{w_2^1}),\\
&&\sqrt{-1}\partial\bar{\partial}\mathbb{K}_3(\lambda_1, \lambda_2, \lambda_3)
=(\sqrt{-1})^2(\lambda_2^2+\lambda_3^2-\lambda_1^2)(w_2^1\wedge\overline{w_2^1}\wedge w_3^1\wedge\overline{w_3^1}\\
&&~~~~~~~~~~~~~~~~~~~~~~~~~~~~~~~~~~~~~~~~~-w_3^1\wedge\overline{w_3^1}\wedge\overline{w_3^2}\wedge w_3^2+\overline{w_3^2}\wedge w_3^2\wedge w_2^1\wedge\overline{w_2^1}),\\
&&\sqrt{-1}\partial\bar{\partial}\mathbb{K}_4(\lambda_1, \lambda_2, \lambda_3)
=(\sqrt{-1})^2(\lambda_1^2+\lambda_3^2-\lambda_2^2)(w_2^1\wedge\overline{w_2^1}\wedge w_3^1\wedge\overline{w_3^1}\\
&&~~~~~~~~~~~~~~~~~~~~~~~~~~~~~~~~~~~~~~~~~+w_3^1\wedge\overline{w_3^1}\wedge w_3^2\wedge\overline{w_3^2}-w_3^2\wedge\overline{w_3^2}\wedge w_2^1\wedge\overline{w_2^1}).
\end{eqnarray*}

\bigskip
\footnotesize
\noindent\textit{Acknowledgments.}
Part of the work was done while the second author was visiting Laboratory of Mathematics for Nonlinear Science, Fudan University; he would like to  thank them for the warm hospitality and supports. Fu is supported in part by NSFC  10831008 and  11025103. Zhou is supported in part by NSFC 11501505.


\begin{thebibliography}{SK}




\normalsize
\baselineskip=17pt


\bibitem[1]{ADM} D. Ali, J. Davidov, O. Muskarov: Compatible almost complex structures on twistor spaces and their Gray-Hervella classes. J. Geom. Phy.
\textbf{75}, 213--229 (2014)

\bibitem[2]{AG} V. Apostolov, P. Gauduchon: The Riemannian Goldberg-Sachs theorem. Internat. J. Math. \textbf{8(4)}, 421--439 (1997)

\bibitem[3]{AHS} M. F. Atiyah, N. J. Hitchin, I. M. Singer: Self-duality in four-dimensional Riemannian geometry. Pro. Roy. Soc. London Ser. A. \textbf{362}, 425--461 (1978)

\bibitem[4]{BO} L. Berard Bergery, T. Ochiai: On some generalization of the construction of twistor
spaces. In: Global Riemannian geometry (T. J. Willmore, N. J. Hitchin, eds.) Ellis Horwood, 52--59 (1984)

\bibitem[5]{Bes} A. L. Besse: Einstein manifolds. Classics in Mathematics. Reprint of the 1987 Edition, Springer-Verlag, Berlin-Heidelberg (2008)

\bibitem[6]{BH} A. Borel, F. Hirzebruch: Characteristic classes and homogeneous spaces I. Amer. J. Math. \textbf{80}, 458--538 (1958)

\bibitem[7]{Bry} R. L. Bryant: Lie groups and twistor spaces. Duke Math. J. \textbf{52(1)}, 223--261 (1985)

\bibitem[8]{Cam} F. Campana: The class $\mathscr{C}$ is not stable by small deformations. Math. Ann. \textbf{290}, 19--30 (1991)

\bibitem[9]{DGM} J. Davidov, G. Grantcharov, O. Mu\v{s}karov: Curvature properties of the Chern connection of Twistor spaces. Rocky Mountain J. Math. \textbf{39}, 27--48 (2009)

\bibitem[10]{DM} J. Davidov, O. Mu\v{s}karov: On the Riemannian curvature of a twistor space. Acta Math. Hungar. \textbf{58}, 319--332 (1991)

\bibitem[11]{BN} P. de Bartolomeis, A. Nannicini: Introduction to differential geometry of twistor spaces. In: Geometric Theory of Singular Phenomena in Partial Differential Equations (Cortona, 1995), Sympos. Math.  XXXVIII, Cambridge Univ. Press, Cambridge, 91--160 (1998)

\bibitem[12]{Des} G. Deschamps: Compatible complex structures on twistor space. Ann. Inst. Fourier, Grenoble \textbf{61(6)}, 2219--2248 (2011)

\bibitem[13]{DLM} G. Deschamps,, N. Le Du, C. Mourougane: Hessian of the natural Hermitian form on twistor spaces. Bull. Soc. Math. France \textbf{145(1)}, 1--27 (2017)

\bibitem[14]{ES} J. Eells, S. Salamon: Twistorial construction of harmonic maps of surfaces into four-manifolds. Ann. Scuola Norm. Sup. Pisa Cl. Sci. \textbf{12}, 589--640 (1985)

\bibitem[15]{FP} J. Fine, D. Panov: Symplectic Calabi-Yau manifolds, minimal surfaces and the hyperbolic geometry of the conifold. J. Diff. Geom. \textbf{82}, 155--205 (2009)

\bibitem[16]{FK} Th. Friedrich, H. Kurke: Compact four-dimensional self-dual Einstein manifolds with positive scalar curvature.
Math. Nachr. \textbf{106},  271--299 (1982)

\bibitem[17]{FWW} J. X. Fu, Z. Z. Wang, D. M. Wu: Semilinear equations, the $\gamma_k$ function, and generalized Gauduchon metrics. J. Eur. Math. Soc. \textbf{15}, 659--680 (2013)

\bibitem[18]{FZ} J. X. Fu, X. C. Zhou: Twistor geometry of Riemannian 4-manifolds by moving frames. Comm. Anal. Geom. \textbf{23(4)}, 819--839 (2015)

\bibitem[19]{Gau} P. Gauduchon: Hermitian connections and Dirac operators. Boll. Un. Mat. Ital. (7) \textbf{11-B}, supl. fasc. 2, 257-288 (1997)

\bibitem[20]{Hit} N. J. Hitchin: K\"{a}hlerian twistor spaces. Proc. London Math. Soc. \textbf{43}, 133--150 (1981)

\bibitem[21]{IP} S. Ivanov, G. Papadopoulos: Vanishing theorems and string Backgrounds. Class. Quant. Grav. \textbf{18}, 1089--1110 (2001)

\bibitem[22]{JR} G. R. Jensen, M. Rigoli: Twistor and Gauss lifts of surfaces in four-manifolds. Contemp. Math. \textbf{101}, 197--232 (1989)

\bibitem[23]{KN} S. Kobayashi, K. Nomizu: Foundations of Differential Geometry I, II. Interscience Publ., (1963, 1969)

\bibitem[24]{KT} D. Kotschick, S. Terzi\'{c}: Chern numbers and geometry of partial flag manifolds. Comment. Math. Helv. \textbf{84}, 587--616 (2009)

\bibitem[25]{LP} C. LeBrun, Y. S. Poon: Twistors, K\"{a}hler manifolds, and bimeromorphic geometry. II. J. Amer. Math. Soc. \textbf{5(2)}, 317--325 (1992)

\bibitem[26]{LY} K. F. Liu, X. K. Yang: Ricci curvatures on Hermitian manifolds. Trans. Amer. Math. Soc. \textbf{369}, 5157--5196 (2017)

\bibitem[27]{Mic} M. L. Michelsohn: On the existence of special metrics in complex geometry. Acta Math. \textbf{149}, 261--295 (1982)

\bibitem[28]{Mu} O. Mu\v{s}karov: Structure presque hermitiennes sur espaces twistoriels et leur types. C. R. Acad. Sci. Paris S\'{e}r. I Math. \textbf{305}, 307--309 (1987)

\bibitem[29]{OR} N. R. O'Brian, J. H. Rawnsley: Twistor spaces. Ann. Global Anal. Geom. \textbf{3(1)}, 29--58 (1985)

\bibitem[30]{PT} C. K. Peng, Z. Z. Tang: Twistor bundle theory and its application. Science in China Series A: Math. \textbf{47(4)}, 605--616 (2004)

\bibitem[31]{Sa1} S. Salamon: Quaternionic K\"{a}hler manifolds. Invent. Math. \textbf{67}, 143--171 (1982)

\bibitem[32]{Sa2} S. Salamon: Riemannian Geometry and Holonomy Groups. Pitman Research Notes in Mathematics Series \textbf{(201)}. Longman Scientific and Technical, Harlow (1989)

\bibitem[33]{Va1} I. Vaisman: Some curvature properties of complex surfaces. Ann. Mat. Pura Appl. \textbf{132(1)}, 1--18 (1982)

\bibitem[34]{Va2} I. Vaisman: On the Twistor Spaces of Almost Hermitian Manifolds. Ann. Global Anal. Geom. \textbf{16(4)} 335--356(1998)

\bibitem[35]{Ver} M. Verbitsky: Rational curves and special metrics on twistor spaces. Geometry and Topology \textbf{18}, 897--909 (2014)

\bibitem[36]{Voi} C. Voisin: Hodge theory and complex algebraic geometry I. Cambridge Studies in Advanced
Mathematics \textbf{(76)}. Cambridge University Press (2002)

\bibitem[37]{YZ} B. Yang, F. Y. Zheng: On Curvature Tensors of Hermitian Manifolds. arXiv:1602.01189

\bibitem[38]{Yan} K. Yang: Almost Complex Homogeneous Spaces and Their Submanifolds. World Scientific Publ. (1987)


\end{thebibliography}
\end{document}